\documentclass[11pt]{article}

\usepackage{hyperref}
\hypersetup{colorlinks,linkcolor={blue},citecolor={blue},urlcolor={red}}

\usepackage[verbose=true,letterpaper,margin=0.9in]{geometry}

\usepackage[mathscr]{eucal}

\usepackage{amsbsy}

\DeclareMathAlphabet{\mathpzc}{OT1}{pzc}{m}{it}

\usepackage{fancyheadings}

\usepackage{enumitem}
\usepackage{float}

\usepackage{color}

\usepackage{amsthm}
\usepackage{amsfonts}
\usepackage{amsmath}
\usepackage{amssymb}

\usepackage{algorithm}

\usepackage{mathabx}

\usepackage{algorithmic}
\usepackage{bbm}

\usepackage{caption}

\usepackage{subcaption}

\usepackage[square, numbers]{natbib}

\usepackage{macros}

\usepackage[noabbrev,nameinlink]{cleveref}

\usepackage{multirow}

\newcommand{\Ind}{\ensuremath{\mathbb{I}}}

\newcommand{\GamMatPlain}{\ensuremath{\mathbf{\Gamma}}}
\newcommand{\GamMat}{\ensuremath{\GamMatPlain_\numobs}}
\newcommand{\Gam}{\GamMat}

\newcommand{\SigMatPlain}{\ensuremath{\mathbf{\Sigma}}}
\renewcommand{\SigMat}{\ensuremath{\SigMatPlain_\numobs}}
\newcommand{\SigHat}{\ensuremath{\widetilde{\SigMatPlain}_\numobs}}

\newcommand{\GamHat}{\widehat{\GamMatPlain}_\numobs}

\newcommand{\LFmat}{\mat{U}}

\newcommand{\Bou}{\ensuremath{b}}

\newcommand{\gamn}[1]{\ensuremath{\gamma_\numobs(#1)}}

\newcommand{\DelHat}{\ensuremath{\widehat{\Delta}}}
\newcommand{\betahat}{\ensuremath{\widehat{\beta}}}
\newcommand{\betastar}{\ensuremath{\paramstar}}

\newcommand{\IdMat}{\ensuremath{\mathbf{I}}}
\newcommand{\ZeroMat}{\ensuremath{\mathbf{0}}}

\newcommand{\E}{\mathbb{E}}

\newcommand{\pollm}{PM$_{2.5}$}
\newcommand{\figdir}{figures}

\newcommand{\alphastar}{\ensuremath{\alpha^*}}
\newcommand{\thetastar}{\ensuremath{\theta^*}}

\newcommand{\bou}{\ensuremath{b}}

\newcommand{\instrumentbar}{\ensuremath{\instrument'}}
\newcommand{\covariatebar}{\ensuremath{\covariate'}}

\newcommand{\Umat}{\ensuremath{\mathbf{U}}}

\newcommand{\thetahat}{\ensuremath{\widehat{\theta}}}

\newcommand{\mkap}{\kappa_\numobs}

\newcommand{\DeltaHat}{\widehat{\Delta}}

\newcommand{\GaussVec}{\ensuremath{G}}

\newcommand{\usedim}{\ensuremath{d}}

\newcommand{\strendo}{\ensuremath{\eta}}
\newcommand{\plainstd}{\ensuremath{\nu}}

\newcommand{\onevecd}{\ensuremath{{\bf{1}}_\usedim}}

\newcommand{\quantile}[1]{\ensuremath{r_{#1}}}

\newcommand{\uvec}{\ensuremath{u}}

\newcommand{\noisehat}{\ensuremath{\widehat{\noise}}}

\newcommand{\shortw}{\ensuremath{W}}

\newcommand{\SigSpec}{\ensuremath{\widetilde{\SigMatPlain}_\numobs}}

\newcommand{\SigEst}{\ensuremath{\widehat{\SigMatPlain}_\numobs}}

\newcommand{\errone}{\ensuremath{e}}
\newcommand{\errtwo}{\ensuremath{\widetilde{e}}}

\newcommand{\vvec}{\ensuremath{v}}
\renewcommand{\vhat}{\ensuremath{\widehat{\vvec}}}

\newcommand{\Dmat}{\ensuremath{\mathbf{D}_\numobs}}

\newcommand{\Qmat}{\ensuremath{\mathbf{Q}_\numobs}}
\newcommand{\Lmat}{\ensuremath{\mathbf{L}_\numobs}}

\newcommand{\strinst}{\ensuremath{\alpha}}

\begin{document}


\begin{center}

  {\bf{\LARGE{Instrumental variables: \\ A non-asymptotic viewpoint}}}
\vspace*{.2in}

{\large{
\begin{tabular}{ccc}
Eric Xia$^{\dagger}$ & Martin
J. Wainwright$^{\dagger, \circ}$ & Whitney Newey$^{\ddagger}$ 
\end{tabular}
}}

\vspace*{.2in}

\begin{tabular}{c}
  Department of Electrical Engineering and Computer
  Sciences$^\dagger$\\
  Department of Economics$^\ddagger$ \\
  Department of Mathematics$^\circ$ \\
  Massachusetts Institute of Technology, Cambridge, MA
\end{tabular}

\medskip

\today

\vspace*{.2in}

\begin{abstract}
We provide a non-asymptotic analysis of the linear instrumental
variable estimator allowing for the presence of exogeneous
covariates. In addition, we introduce a novel measure of the strength
of an instrument that can be used to derive non-asymptotic confidence
intervals.  For strong instruments, these non-asymptotic intervals
match the asymptotic ones exactly up to higher order corrections; for
weaker instruments, our intervals involve adaptive adjustments to the
instrument strength, and thus remain valid even when asymptotic
predictions break down.  We illustrate our results via an analysis of
the effect of \pollm~pollution on various health conditions, using
wildfire smoke exposure as an instrument.  Our analysis shows that
exposure to \pollm~pollution leads to statistically significant
increases in incidence of health conditions such as asthma, heart
disease, and strokes.
\end{abstract}

\end{center}

\section{Introduction}

The consistency of standard least-squares estimates rely on covariates
being uncorrelated with the additive noise; when this condition breaks
down, the covariates are said to be endogenous.  Instrumental
variables (IV) provide an important class of methods for addressing
this challenge. Since their introduction in the appendix of Wright's
book~\cite{wright1928tariff}, IV methods have come to play a key role
in semi-parametric statistics, econometrics, reinforcement learning,
and causal inference.  They exploit randomness in an external
measurement, known as an instrument, in order to overcome endogeneity.
Classical work on IV methods focused on their use for
errors-in-variables~\cite{MR14668}, whereas Angrist and
Krueger~\cite{ak91qob} popularized their use for answering causal
questions.  Econometric questions in which IV methods have proven
valuable include understanding the economic returns of
education~\cite{ak91qob, card1995collegeeduc}; the effect of pre-trial
detention on different outcomes~\cite{frandsen2023judges,
  lesliepope2017}; the consequences of serving in the Vietnam War for
future earnings and mortality ~\cite{angrist90viet,angrist95causal};
and the effects of C-sections on infant health
outcomes~\cite{card2023infant}, among others.  Instrumental variables
also arise in reinforcement learning, where they underlie the
TD-learning family of algorithms (e.g.,~\cite{Sut88,SutBar18,
  KhaPanRuaWaiJor20,DuaWanWai24}).

Given this wide range of applied uses, there is now a rich literature
on the theoretical properties of instrumental variables.  The
papers~\cite{angrist95causal, imbens95LATE} provided an explicit
causal interpretation of the IV estimate (or Wald estimator, to be
precise) as the local average treatment effect. It is well-known that
the IV estimator is a special case of the more general class of
estimators known as generalized method of
moments~\cite{newey94handbook}.  From this, the classical asymptotic
theory of GMM estimators~\cite{hansen92gmm} thus also apply. Modern
work in the theory of IV estimators has focused on the \emph{weak
instrument setting}~\cite{stock97weakiv, andrews19weakiv}, where the
correlation (i.e. strength) of the instrument with the covariate of
interest is too small, relative to the sample size. In another line of
work, there has been a focus on using LASSO~\cite{tibshirani96lasso}
to address high-dimensional issues that can occur while using
instrumental variables~\cite{belloni2012doublelasso}, or incorporating
more modern techniques in machine learning under IV
estimation~\cite{chen2021harmless}, with mixed
success~\cite{angrist2019machinelabor}.

The vast majority of this research has been asymptotic in nature. The
references~\cite{hansen92gmm, newey94handbook} provide a classical
$\sqrt{\numobs}$-asymptotic analysis of instrumental variables through
the central limit theorem. Yet it is well know that in modern
statistical settings, this style of CLT analysis often breaks
down. The weak instrument literature studies a specialized asymptotic
regime that considers the ratios of normal
approximations~\cite{andrews19weakiv}. However recent
work~\cite{young2022iv} has highlighted issues with the normal
approximation used within the weak IV literature. In an effort to
compare the finite-sample behavior of various procedures, researchers
have turned towards higher-order asymptotic expansions and comparing
these higher order terms with each other~\cite{neweysmith2004}. The 
paper~\cite{evdokimov2018inference} provides an asymptotic analysis
of various types of IV estimators in the setting with many 
instruments and exogeneous covariates. In
spite of all this, a non-asymptotic analysis of instrumental variables
still remains lacking. One contribution of this work is the first
instance of this style of analysis, to the best of our knowledge.

To further motivate our paper, as a practical application we consider
the health effects of fine particulate matter (PM$_{2.5}$)
exposure. The association between exposure to high levels of
\pollm~and negative health outcomes such as mortality
rates~\cite{kloog2013pm25mortality, franklin2007association},
childhood asthma rates~\cite{lin2002childhood}, cardiovascular issues
in the elderly~\cite{barnett2006effects}, and many other conditions,
is well-known in the literature. However all of these have studied the
association between exposure to \pollm~and health outcomes. There is a
relatively recent line of work using causal inference/econometric
methodology to analyze this connection, beginning with
Pope~\cite{iii1999epidemiology}. The
paper~\cite{schwartz2015estimating} uses air trajectory data as an
instrument for \pollm; a related work~\cite{wang2016estimating} uses a
differences-in-differences approach. Both studies draw conclusions
consistent with previous studies on the effects of \pollm~exposure and
increasing mortality rates. In the last few years alone, there have
been many studies dedicated to analyzing the health effects of
\pollm~exposure in various Asian nations, using atmospheric data as an
instrument~\cite{xu2022assessing}.

\subsection{Our contributions}

The main contribution of this paper is to provide some non-asymptotic
insight into the behavior of the classical IV estimator, both in terms
of its $\ell_2$-error, and in terms of linear functionals of the IV
estimate.  Moreover, in many settings, researchers often introduce
exogeneous covariates; these covariates often are additional
measurements used as features to improve the quality of our estimate
but are nuisances in the sense that we are not interested in their
effect on the response. We provide guarantees on the estimation of the
component of interest in the presence of exogeneous covariates, which
follows as a special case from our results on linear functionals.

Our first set of results (Theorems~\ref{ThmStandardIV}
and~\ref{ThmLinFuncIV}) are aimed at characterizing the rate at which
the IV error converges to the asymptotic prediction.  The terms in
these bounds depend on various aspects of the problem, including the
number of exogenous variables and the strength of endogeneity, and we
illustrate our theoretical predictions via simulations on synthetic
ensembles.  Our second set of results, including
Theorem~\ref{ThmConfInt} and Corollaries~\ref{CorRefineConfInt},
~\ref{cor:uniform-CI} and~\ref{CorScalarConfInt}, provide procedures
for computing confidence intervals based on the data.  These bounds
involve adaptive correction factors to the classical asymptotic
prediction based on the sandwich estimate.  These results also lead us
to a novel measure of the strength of an instrument.

Finally, to illustrate the use of our findings, we apply instrumental
variables to analyze the effects of \pollm~exposure on negative health
outcomes. For our paper, we have constructed a novel dataset based on
various governmental sources to analyze this relationship. In our
approach, we treat each individual census tract as a unit of
observation, and consider aggregate measures of health outcomes for
these census tracts. The data about \pollm~emissions are taken from
the Public Health Tracking Network under the Centers for Disease
Control and Prevention (CDC)~\cite{pm25data}. We constructed an
instrument based around wildfire exposure, taken from the National
Interagency Fire Center~\cite{firedata}. The data for the exogeneous
covariates, including information about the racial and demographic
makeup, breakdown of employment by industry, log median wage,
education, and much more is taken from the 2020 US Census, accessed
via the IPUMS NHGIS~\cite{censusdata}. The health outcomes of census
tracts are obtained from the PLACES project conducted via the
CDC~\cite{placesdata}.

\subsection{Notation}

For a vector $v \in \RR^d$, we use $\twonorm{v}$ to denote its
Euclidean norm.  For a matrix $\matA \in \RR^{d \times k}$, its
spectral norm is given by
\begin{align*}
\specnorm{\matA} = \sup_{v \in \RR^k} \frac{\abs{v^T \matA
    v}}{\twonorm{v}^2},
\end{align*}
and its maximum and minimum singular values by $\maxsingvalue{\matA}$
and $\minsingvalue{\matA}$, respectively.  We use $\identity_\dims \in
\RR^{\dims \times \dims}$ to denote the identity matrix of size
$\dims$.  We use $[\numobs]$ to denote the set $\{1, 2, \ldots \numobs
\}$.  For a pair of symmetric matrices, we write $\matA \succeq \matB$
to mean that $\matA - \matB$ is positive semidefinite (PSD). For a
matrix $\mat{A} \in \RR^{\dims_1 \times \dims_2}$, we use $\mat{A}^T$
to denotes its matrix transpose in $\RR^{\dims_2 \times \dims_1}$. We
use $\overset{p}{\to}$ to denote convergence in probability, and
$\rightsquigarrow$ to indicate convergence in distribution.
Throughout the paper, we use $c_0, c_1, c_2, \ldots$ as universal
constants whose values may change from line to line.

\subsection{Paper organization}

The remainder of the paper is organized as follows.  We begin
in~\Cref{sec:background} with background on the linear IV estimator,
and its extension to exogenous variables.  In~\Cref{SecMain}, we turn
to the main results of the paper, including non-asymptotic bounds for
the standard IV estimate (\Cref{ThmStandardIV}
in~\Cref{sec:stand-iv-results}), along with bounds for linear
functions allowing for exogenous covariates (\Cref{ThmLinFuncIV}
in~\Cref{sec:exog-cov}).  \Cref{sec:comp-conf} is devoted to
non-asymptotic and fully computable confidence intervals.
In~\Cref{SecApplications}, we explore the applied consequences of our
results, including a numerical study of our confidence intervals,
along with an applied study of IV methods for assessing the effects of
PM2.5 pollution.  We conclude with a discussion
in~\Cref{sec:discussion}.


\section{Background}
\label{sec:background}

Here we provide a very brief introduction to the theory of
instrumental variables; see the
sources~\cite{angrist2009mostly,newey94handbook} for more details.
\Cref{SecStandIV} introduces the standard instrumental variable setup,
whereas~\Cref{SecExoIV} introduces the notion of exogeneous
covariates.


\subsection{Standard instrumental variables}
\label{SecStandIV}

Consider a scalar response $\response$ and covariate vector
$\covariate \in \real^d$ linked via the linear model
\begin{subequations}
\begin{align}
\label{eqn:linear-model}
\response & = \inprod{\covariate}{\paramstar} + \noise,
\end{align}
where $\noise$ is an additive noise term.  Ordinary least-squares
(OLS) is a standard way of estimating the unknown vector $\paramstar
\in \real^d$ of parameters.  In the classical formulation, the noise
$\noise$ is assumed to be a zero-mean random variable such that
$\Exs[\covariate \noise] = 0$.  This lack of correlation ensures
consistency of the OLS estimate.

Instrumental variables are designed to handle cases in which the
covariate vector $\covariate$ and noise term $\noise$ are correlated.
Such dependence can arise for various reasons, with an archetypal
example being mis-specification in a linear model.  As a simple but
concrete illustration, consider a linear model with covariates $(X, u)
\in \real^{d} \times \real$ linked to the response $\response$ via
\begin{align}
\label{EqnWell}  
\response & = \inprod{\covariate}{\paramstar} + u \alphastar + w.
\end{align}
\end{subequations}
When this linear model is well-specified, then we are guaranteed to
have the orthogonality condition \mbox{$\Exs[(X, u) w] = 0$.}
However, suppose that the additional covariate $u$ is not observed
(and hence not modeled); in this case, it is natural to view the
model~\eqref{EqnWell} as a version of the $\covariate$-based
model~\eqref{eqn:linear-model} with the augmented noise term $\noise
\defn u \alphastar + w$.  An easy calculation gives $\Exs[X \noise] =
\Exs[X u] \alphastar$, so that whenever $\alphastar \neq 0$, any
correlation between $X$ and $u$ will render $X$ and $\noise$
correlated as well.  In the econometrics literature, this phenomenon
is known as \emph{omitted variable bias}. \\

Returning to the original model~\eqref{eqn:linear-model}, a vector $Z
\in \real^d$ is said to be an instrument if
\begin{align}
\label{eqn:instrument-condition}
\EE[\instrument \noise ] \stackrel{(\star)}{=} 0, \quad \text{and the
  matrix $\GamMatPlain \defn \EE\left[\instrument \covariate^T
    \right]$ is full rank.}
\end{align}
Condition ($\star$) is known as the \emph{clean} instrument condition,
whereas the full-rank condition is known as the \emph{fully
correlated} requirement.  Given an instrument $Z$, a straightforward
calculation using the linear model~\eqref{eqn:linear-model} shows that
$\betastar$ satisfies
\begin{align}
\Exs[Z \response] = \Exs[Z X^T] \betastar = \GamMatPlain \betastar,
\quad \mbox{or equivalently $\betastar = \GamMatPlain^{-1} \Exs[Z
    \response]$,}
\end{align}
where the second equation uses the full rank condition.  While the
expectations defining this relation are unknown, given i.i.d. samples
$\{(\response_i, \covariate_i, Z_i) \}_{i=1}^\numobs$, we can use the
plug-in principle to form the estimate
\begin{align}
\label{eqn:IV-estimator}
\paramhat = \left(\empcrosscov{\instrument}{\covariate^T} \right)^{-1}
\left(\empcrosscov{\response}{\instrument}\right),
\end{align}
which is the standard instrumental variable estimator.

Under i.i.d. sampling and regularity conditions, the standard
asymptotic argument shows that
\begin{align}
\label{eqn:asymptotic-guarantee}
\sqrt{\numobs} (\paramhat - \paramstar) \rightsquigarrow
\mathcal{N}(0, \GamMatPlain^{-1} \SigMatPlain \GamMatPlain^{-T})
\end{align}
where $\GamMatPlain \defn \EE[\instrument \covariate^T]$ and
$\SigMatPlain \defn \EE[\noise^2 \instrument \instrument^T]$.


\subsection{Instrumental variables and exogenous covariates}
\label{SecExoIV}

In most practical settings, in addition to a set of endogenous
covariates $\covariate \in \real^\usedim$, we also have a vector
\emph{exogenous covariates} $\exogcov \in \RR^{\ndims}$, leading to
the augmented model
\begin{align}
\label{eqn:exog-iv-model}
\response &= \inprod{\covariate}{\paramstar} +
\inprod{\exogcov}{\alphastar} + \noise_i,
\end{align}
with the moment restrictions $\EE[\instrument\noise ] = 0$ and
$\EE[\exogcov\noise] = 0$.  Here our primary goal remains to estimate
$\paramstar$, and the additional covariates $\exogcov$ are introduced
to avoid mis-specification and/or reduce the variability in our
estimate of $\paramstar$.  The associated parameter vector
$\alphastar$ can be viewed as a ``nuisance''; it is not of
interest \emph{per se}, but needs to be handled so as to
obtain a better estimate of
$\paramstar$.

Let us now describe how the problem of estimating $\paramstar$ can be
reformulated as one of estimating a linear function of the parameter
vector in a ``lifted'' standard IV estimate.  We define the full
parameter vector $\thetastar \defn \begin{bmatrix} \betastar
  \\ \alphastar \end{bmatrix} \in \real^{d + p}$, along with the lifted
covariate and instrument vectors
\begin{align*}
  \covariatebar \defn \begin{bmatrix} \covariate
    \\ \exogcov \end{bmatrix}, \quad \text{and} \quad \instrumentbar
  \defn \begin{bmatrix} \instrument \\ \exogcov \end{bmatrix}.
\end{align*}
With this notation, our original model can be written more compactly
as $\response = \inprod{\covariatebar}{\thetastar} + \noise$.  Note
that we have have $\EE[\instrumentbar \noise] = 0$ by construction, so
that the lifted vector $\instrumentbar \in \real^{p +d }$ is a valid
instrument.  Consequently, we can compute the standard IV estimate
$\thetahat$ using samples of the triples $(y_i, \covariatebar_i,
\instrumentbar_i)$, and we can recover the estimate $\betahat$ of
interest as $\betahat = \Umat^T \thetahat$, where $\Umat^T
\defn \begin{bmatrix} \IdMat_d & \mathbf{0}_p
\end{bmatrix}$.  We give results on the estimation of such linear
functions of an IV estimate in~\Cref{ThmLinFuncIV}
and~\Cref{CorRefineConfInt} to follow in the sequel.


\section{Main results}
\label{SecMain}

In this section we present the main theorems of this
paper. \Cref{sec:stand-iv-results} is devoted to non-asymptotic
guarantees for the standard instrumental variable
estimator. \Cref{sec:exog-cov} provides extensions to estimation in
the presence of exogeneous covariates, which amounts to estimating a
linear mapping of the original IV estimate. Finally,
in~\Cref{sec:comp-conf}, we turn to the construction of confidence
intervals that can be equipped with non-asymptotic guarantees.


\subsection{Non-asymptotic bounds for standard IV}
\label{sec:stand-iv-results}

We begin with some finite-sample results for the standard IV estimate
without exogenous covariates.  In particular, suppose that we observe
triples $\{(\response_i, \covariate_i, \instrument_i)
\}_{i=1}^\numobs$ such that
  \begin{align}
\label{EqnIVModel}
  \response_i = \inprod{\covariate_i}{\paramstar} + \noise_i \qquad
  \mbox{with $\EE[\instrument_i \noise_i] = 0$.}
  \end{align}  
Defining the matrix $\Gam \defn \empsum{i} \EE[\instrument_i
  \covariate_i^T]$, we assume that it is invertible, as is necessary
for the instrument sequence $\{\instrument_i\}_{i=1}^\numobs$ to be
useful.  Apart from this condition, we impose \emph{no other}
distributional conditions on the covariates $\{\covariate_i
\}_{i=1}^\numobs$.

Our main assumptions are imposed on the zero-mean random vectors
$\instrument_i \noise_i \in \real^\usedim$: in particular, we require
the existence of a finite constant $\Bou$ such that
\begin{subequations}
\begin{align}
\label{EqnIVInd}
\mbox{$\{\instrument_i \noise_i \}_{i=1}^\numobs$ is an independent
  sequence of zero-mean variables, and} \\
\label{EqnIVBound}
\|\Gam^{-1} (\instrument_i \noise_i) \|_2 \leq \Bou \qquad
\mbox{almost surely for each $i = 1, \ldots, \numobs$.}
\end{align}
\end{subequations}
We use the independence~\eqref{EqnIVInd} and $\bou$-boundedness
conditions~\eqref{EqnIVBound} in order to establish tail bounds on
sums of the terms $\Gam^{-1}(\noise_i \instrument_i)$.  We note that
in many practical settings of interest, the instruments
$\instrument_i$ and noise $\noise_i$ are bounded, in which case the
$\bou$-boundedness condition holds.  Moreover, we note that that it is
straightforward (although requiring more technical analysis) to
extend our results to unbounded random vectors with reasonable tail
behavior. \\

\subsubsection{A non-asymptotic bound}

Our result involves the matrices
\begin{subequations}
\begin{align}
\label{EqnDefnGamma}    
\GamHat \defn \empcrosscov{\instrument}{\covariate^T}, \quad
\mbox{and} \quad \Gam \defn \Exs[\GamHat] \equiv \empsum{i}
\EE[\instrument_i \covariate_i^T], \\
\label{EqnDefnSigMat}
\SigSpec \defn \empsum{i} \noise_i^2 \instrument_i \instrument_i^T,
\quad \mbox{and} \quad \SigMat \defn \Exs[\SigSpec] \equiv
\empsum{i} \EE[\noise_i^2 \instrument_i \instrument_i^T].
  \end{align}
\end{subequations}
We also define the zero-mean random vector
\begin{align}
  \label{EqnDefnGaussVec}
 \GaussVec_\numobs & \defn \frac{1}{\sqrt{\numobs}} \sum_{i=1}^\numobs
 \GamMat^{-1} \instrument_i \noise_i,
\end{align}
and observe that we have \mbox{$\cov(\GaussVec_\numobs) = \GamMat^{-1}
  \SigMat \GamMat^{-T}$} by independence of the sequence
$\{\instrument_i \noise_i \}_{i=1}^\numobs$. Thus, given the classical
asymptotic behavior~\eqref{eqn:asymptotic-guarantee} of the standard
IV estimate under i.i.d. sampling, it is reasonable to expect that the
IV estimate---under the more relaxed distribution assumptions that we
impose---should satisfy
\begin{align}
\label{EqnOwen}
\sqrt{\numobs} \; \Exs \|\betahat - \betastar\|_2 & \approx \Exs
\|\GaussVec_\numobs\|_2
\end{align}
for all \emph{suitably large} $\numobs$.  The goal of this section is
to make this intuition precise: we prove an upper bound, valid for
\emph{any} sample size $\numobs$, in which this quantity is the
leading order term. \\

In order to establish a non-asymptotic bound of this type, we need to
introduce a quantity that measures the deviation between the
population matrix $\Gam$ and its empirical counterpart $\GamHat$ from
equation~\eqref{EqnDefnGamma}.  In particular, we define
\begin{align}
\gamn{\Gam} \defn \specnorm{\GamHat^{-1} \Gam - \Id_\dims}
\end{align}
where $\specnorm{\cdot}$ denotes the $\ell_2$-operator norm, or
maximum singular value.  As we discuss below, in typical settings, we
expect that $\gamn{\Gam}$ scales as $\sqrt{\frac{d}{\numobs}}$ with
the dimension $d$ and sample size $\numobs$. \\

\noindent With this set-up, we are now ready to state our first main
result:
\begin{theorem}
\label{ThmStandardIV}
Consider the model~\eqref{EqnIVModel} such that the sequence
$\{\instrument_i \noise_i\}_{i=1}^\numobs$ satisfies the
independence~\eqref{EqnIVInd} and $\bou$-boundedness
conditions~\eqref{EqnIVBound}.  For any $\pardelta \in (0, 1)$, we
have
\begin{align}
\label{EqnMartin}
\twonorm{\paramhat - \paramstar} & \leq \frac{1 +
  \gamn{\Gam}}{\sqrt{\numobs}} \left \{ \Exs \|\GaussVec_\numobs\|_2 +
\sqrt{2 \log(\tfrac{1}{\pardelta}) \EE\specnorm{\Gam^{-1} \SigSpec
    \Gam^{-T}}} + \frac{3\Bou
  \log(\tfrac{1}{\pardelta})}{\sqrt{\numobs}} \right \}
\end{align}
with probability at least $1 - \pardelta$.
\end{theorem}
\noindent We prove this result in~\Cref{SecProofNonAsymp}.  The proof
involves a careful decomposition of the error in the IV estimate, and
then reducing the problem to studying the supremum of a certain
empirical process. 

\paragraph{Some remarks:}  Let us make a few comments so as to interpret the
bound~\eqref{EqnMartin}. Beginning with the leading term, by Jensen's
inequality, we have
\begin{align*}
\Exs \|\GaussVec_\numobs\|_2 & \leq \sqrt{ \Exs
  \|\GaussVec_\numobs\|_2^2} \; = \; \sqrt{ \trace(\GamMat^{-1}
  \SigMat \GamMat^{-T})},
\end{align*}
where the last step uses the fact that $\cov(\GaussVec_\numobs) =
\GamMat^{-1} \SigMat \GamMat^{-T}$ by construction.  Given the
classical asymptotic relation~\eqref{eqn:asymptotic-guarantee}, the
appearance of this covariance matrix is to be expected.

Turning to the second term, let us first consider the uni-variate
setting ($\dims = 1$). In this case, we have $\EE\specnorm{\Gam^{-1}
  \SigSpec \Gam^{-T}} = \trace(\Gam^{-1} \SigMat \Gam^{-T})$ so
the bound simplifies to
\begin{align*}
\twonorm{\paramhat - \paramstar} \leq \frac{1 +
  \gamn{\Gam}}{\sqrt{\numobs}} \left( \big(1 +
\sqrt{2\log(\tfrac{1}{\pardelta}}) \big) \cdot \sqrt{\trace(\Gam^{-1}
  \SigMat \Gam^{-T})} + \frac{3\Bou
  \log(\tfrac{1}{\pardelta})}{\sqrt{\numobs}} \right).
\end{align*}
For an arbitrary dimension $\usedim$, since $\EE\specnorm{\Gam^{-1}
  \SigSpec \Gam^{-T}} \leq \trace(\Gam^{-1} \SigMat \Gam^{-T})$, this
inequality is always valid.  However, in high-dimensional settings,
the trace of a matrix can be be considerably larger---up to factor of
$\usedim$---than its spectral norm, so that the bound that we have
established can be much sharper.

Finally, let us consider the pre-factor $\gamn{\Gam} =
\specnorm{\GamHat^{-1} \Gam - \Id_\dims}$.  In a classical
analysis---viewing the dimension $\dims$ as fixed---this term scales
as $1/\sqrt{\numobs}$ so that we have $1 + \gamn{\Gam} = 1 +
O(1/\sqrt{\numobs})$.  Thus, as $\numobs$ tends to infinity, other
problem parameters fixed, our bound matches the heuristic
prediction~\eqref{EqnOwen} that motivated our analysis.

\subsubsection{Numerics via the non-asymptotic lens}

As noted earlier, our main interest is not re-capturing the asymptotic
prediction, but rather understanding non-asymptotic aspects of the IV
estimate.  In order to highlight some non-asymptotic predictions
of~\Cref{ThmStandardIV}, it is helpful to perform simulations on
synthetic ensembles designed to reveal certain behavior.  In
particular, here we explore the effects of two parameters that can
vary: (i) the degree of endogeneity; and (ii) the number of
instruments $\usedim$.

\paragraph{Degree of endogeneity:}
We begin by describing a simple ensemble to allows us to investigate
the effect of varying endogeneity.  Suppose that we generate covariate
vectors $\covariate \in \real^\usedim$ that are related to the
instrument $\instrument \in \real^\usedim$ and additive noise $\noise
\in \real$ via the equation
\begin{align}
  \label{EqnEndo}
  \covariate & = \strinst \instrument + \strendo \noise \bf{1}_\usedim
  + \plainstd W.
\end{align}
Here $W \sim \mathcal{N}(0, \IdMat_\usedim)$ is a second source of
noise, independent of the pair $(\instrument, \noise)$, and $\onevecd
\in \real^\usedim$ denotes a vector of all-ones.  By construction, we
have $\E [\covariate \noise] = \strendo \onevecd$, so that parameter
$\strendo$ controls the degree of endogeneity.

In order to study the effect of increasing endogeneity, we study
ensembles with fixed instrument strength $\strinst = 1$ and noise
level $\plainstd = 0$, and endogeneity level $\strendo$ varying over
the interval $[0, 3.0]$.  With this set-up, the setting $\strendo = 0$
leads to ensembles with $X = Z$, in which case the IV estimate reduces
to the OLS estimate.  Increasing $\strendo > 0$ leads to problems
where $X$ is endogenous, with the parameter $\strendo$ a measure of
its strength.  We constructed ensembles in dimension $\usedim = 5$ and
parameters chosen to ensure that the asymptotic rescaled MSE $\numobs
\Exs \|\betahat - \betastar \|_2^2$ was equal to $1$.
\begin{figure}[h]
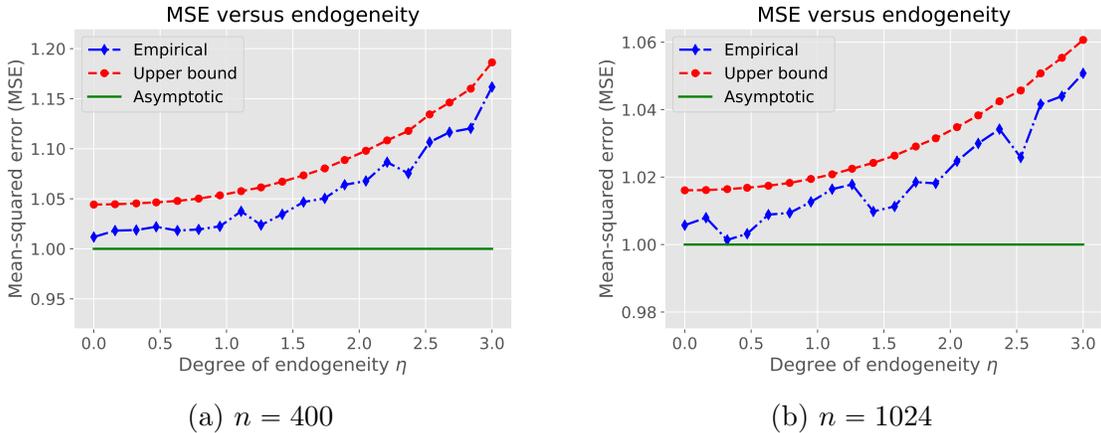

  \begin{center}
    \begin{tabular}{ccc}
      \widgraph{0.45\textwidth}{\figdir/fig_confound_n400_numtrials30000}
      &&
      \widgraph{0.45\textwidth}{\figdir/fig_confound_n1024_numtrials30000}
      \\ (a) $\numobs = 400$ && (b) $\numobs = 1024$
    \end{tabular}      
\caption{Plot of the rescaled mean-squared error (MSE) $\numobs \: \E
  \|\betahat - \betastar\|_2^2$ versus the degree of endogeneity
  $\strendo$. Experiments were performed on ensembles from the
  generative model~\eqref{EqnEndo} with dimension $\usedim = 5$, and
  sample sizes $\numobs = 400$ and $\numobs = 1024$ in panels (a) and
  (b), respectively, with the parameter $\strendo$ controlling the
  degree of endogeneity.  For each ensemble, the asymptotic rescaled
  MSE was equal to $1$, as marked in a solid green line.  The
  empirical MSE was estimated on the basis of $M = 30000$ Monte Carlo
  trials, whereas the theoretical upper bound was computed using the
  $1 + \gamn{\Gam}$ correction to the leading term
  in~\Cref{ThmStandardIV}.}
\label{FigDegreeEndo}  
\end{center}
\end{figure}

\Cref{FigDegreeEndo} plots the rescaled MSE $\numobs \Exs \|\betahat -
\betastar\|_2^2$ versus the endogeneity parameter $\strendo$ for two
different sample sizes ($\numobs = 400$ in panel (a), and $\numobs =
1024$ in panel (b)).  Note that the discrepancy between the asymptotic
prediction and the empirical MSE of the IV estimate (estimated by
Monte Carlo trials) increases as $\strendo$ grows, and the theoretical
upper bound from~\Cref{ThmStandardIV} tracks this behavior.

\paragraph{Growing numbers of instruments:}
In applications of IV, it can be the case that the number of
instruments $\usedim$ is relatively ``large'' compared to the sample
size $\numobs$, in which case the asymptotic predictions might be
inaccurate.  Let us explore some dimensional effects in the context
of~\Cref{ThmStandardIV}.  We begin by observing that under mild
conditions on the pairs $(\covariate_i, \instrument_i)$, standard
random matrix theory (cf. Chapter 5 in the book~\cite{dinosaur2019})
can be used to show that
\begin{align*}
\gamn{\Gam} = \specnorm{\GamHat^{-1} \Gam - \Id_\dims} \asymp
\sqrt{\tfrac{\usedim}{\numobs}}
\end{align*}
with high probability.  For problem ensembles with this behavior, we
we would expect that the MSE should converge to the asymptotic
prediction for $(\numobs, \usedim)$-sequences such that
$\tfrac{\usedim}{\numobs} \rightarrow 0$.

\begin{figure}[h]
  \begin{center}
    \begin{tabular}{c}
      \widgraph{0.45\textwidth}{\figdir/fig_euclid_sublin_numtrials2000}
    \end{tabular}      
\caption{Plot of the rescaled mean-squared error (MSE)
  $(\numobs/\usedim) \E \|\betahat - \betastar\|_2^2$ versus the
  sample size $\numobs$ for ensembles with the dimension
  $\usedim_\numobs$ growing as $\usedim_\numobs = \lceil
  \numobs^{0.30} \rceil$.  The empirical MSE was estimated on the
  basis of $M = 2000$ Monte Carlo trials, whereas the theoretical
  upper bound was computed using the $1 + \gamn{\Gam}$ correction to
  the leading term in~\Cref{ThmStandardIV}.}
\label{FigGrowInstrument}  
\end{center}
\end{figure}

In order to verify this prediction, we generated $(\covariate,
\instrument)$ pairs from the model~\eqref{EqnEndo} with $\strinst =
\strendo = \plainstd = 1$; the noise variable $\noise$ being standard
Gaussian; and the instrument $\instrument \in \real^{\usedim}$ with
i.i.d. Rademacher entries (equiprobable random signs in $\{-1, +1
\}$).  We studied ensembles parameterized by the pair $(\numobs,
\usedim)$, and let the number of instruments grow as
\mbox{$\usedim_\numobs = \lceil \numobs^{0.30} \rceil$.}  As shown
in~\Cref{FigGrowInstrument}, both the bound from~\Cref{ThmStandardIV}
and the empirical MSE of the IV estimator converge to the asymptotic
prediction under this form of high-dimensional scaling.

    
\subsection{Bounds for linear functions and exogenous covariates}
\label{sec:exog-cov}

Recall from~\Cref{SecExoIV} the more general set-up in which the
linear model includes a combination of endogenous covariates
$\covariate \in \real^d$ and exogenous covariates $\exogcov \in
\real^p$, and our goal is to estimate the parameters associated with
the endogenous covariates.  In terms of the sampling model, suppose
that we observe triples $(\response_i, \covariate_i, \exogcov_i,
\instrument_i)$ from the linear model
\begin{align}
\label{EqnIVModelExo}
\response_i & = \inprod{\covariate_i}{\paramstar} +
\inprod{\exogcov_i}{\alphastar} + \noise_i, \qquad \mbox{such that
  $\EE[\instrument_i \noise_i] = 0$ and $\EE[\exogcov_i \noise_i] =
  0$.}
\end{align}

As described in ~\Cref{SecExoIV}, the problem of estimating
$\paramstar$ can be reformulated as an instance of estimating a linear
mapping of the parameter vector in a standard IV model.  We recall
the notation
\begin{align*}
  \thetastar \defn (\paramstar, \alphastar), \quad \covariatebar_i
  \defn (\covariate_i, \exogcov_i), \quad \mbox{and} \quad
  \instrumentbar_i \defn (\instrument_i, \exogcov_i).
\end{align*}
Using this shorthand, we have a more compact reformulation of our
problem: we observe triples $(\response_i, \covariatebar_i,
\instrumentbar_i)$ such that
\begin{subequations}
\begin{align}
\label{EqnIVModelExoCompact}
\response_i & = \inprod{\covariatebar_i}{\thetastar} + \noise_i,
\qquad \mbox{such that $\EE[\instrumentbar_i \noise_i] = 0$.}
\end{align}
The IV estimate is given by $\thetahat = (\frac{1}{\numobs}
\sum_{i=1}^\numobs \instrumentbar_i (\covariatebar_i)^T)^{-1}
\sum_{i=1}^\numobs \instrumentbar_i \response_i$, and our goal is to
estimate the parameter $\betastar = \begin{bmatrix} \IdMat_d &
  \ZeroMat_p
  \end{bmatrix}  \thetastar$,
where $\IdMat_d$ denotes a $d$-dimensional identity matrix, and
$\ZeroMat_p$ denotes a $p$-dimensional matrix of zeroes.  

This problem is special case of estimating the quantity $\Umat^T
\thetastar$ for some matrix $\Umat \in \real^{D \times k}$ with $D
\defn d + p$.  In other applications, we might be interested in
estimating a single co-ordinate of $\thetastar$, so that $\Umat \in
\real^D$ is a column vector.  The theory below encompasses all of
these cases.

\subsubsection{A non-asymptotic bound}

With a slight abuse of notation, we use $(\Gam, \GamHat)$ and
$(\SigMat, \SigSpec)$ to denote the population-empirical matrices
defined as in equations~\eqref{EqnDefnGamma}
and~\eqref{EqnDefnSigMat}, with the pair $(\covariate_i,
\instrument_i)$ replaced by $(\covariatebar_i, \instrumentbar_i)$.  In
terms of tail conditions on the random vectors, we impose the
following $\Bou$-boundedness conditions
\begin{align}
\label{EqnIVBoundExo}  
\|\Gam^{-1} ( \instrumentbar_i \noise_i) \|_2 \leq \Bou, \quad
\mbox{and} \quad \|\LFmat^T \Gam^{-1}(\instrumentbar_i \noise_i)\|_2
\leq \Bou
\end{align}
\end{subequations}
almost surely for each $i = 1, \ldots, \numobs$.  As noted earlier,
these assumptions could be relaxed to sub-Gaussian tail conditions, at
the expense of more complicated arguments and/or truncation arguments.

Recall from equation~\eqref{EqnDefnGaussVec} the random vector
$\GaussVec_\numobs \defn \frac{1}{\sqrt{\numobs}} \sum_{i=1}^\numobs
\GamMat^{-1} \instrument_i \noise_i$.  For a general matrix $\Amat \in
\real^{D \times k}$, we define the functional
\begin{align}
\Psi_\numobs(\Amat; \delta) & \defn \Exs[ \|\Amat^T
  \GaussVec_\numobs\|_2] + \sqrt{2\log(\tfrac{2}{\pardelta}) \EE
  \specnorm{\Amat^T \Gam^{-1} \SigSpec \Gam^{-T} \Amat }} + \frac{3\Bou
  \log(\tfrac{2}{\pardelta})}{\sqrt{\numobs}} 
\end{align}
Our main result involves this functional with $\Amat = \Umat$ in the
leading order term, and $\Amat = \IdMat_D$ in the second-order term.

\begin{theorem}
\label{ThmLinFuncIV}
Consider the model~\eqref{EqnIVModelExoCompact}, and suppose that the
sequence $\{\instrumentbar_i \noise_i \}_{i=1}^\numobs$ is independent
and satisfies the $(\Umat, \bou)$-boundedness
condition~\eqref{EqnIVBoundExo}. Then for any $\pardelta \in (0,1)$,
the IV estimate $\thetahat$ satisfies the bound
\begin{align}
\label{EqnLinFuncIV}  
\twonorm{\LFmat^T (\thetahat - \thetastar)} &\leq
\frac{1}{\sqrt{\numobs}} \Psi_\numobs(\Umat; \delta) +
\frac{\specnorm{\LFmat^T (\GamHat^{-1} \Gam -
    \Id_\dims)}}{\sqrt{\numobs}} \Psi_\numobs(\IdMat_D; \delta)
\end{align}
with probability at least $1 - \pardelta$.
\end{theorem}
\noindent See~\Cref{SecProofNonAsymp} for a proof of this result. \\

We begin by observing that this claim is the natural generalization
of~\Cref{ThmLinFuncIV} on the standard IV estimate.  In particular, if
we set $\Umat = \IdMat_D$, then the bound~\eqref{EqnLinFuncIV} reduces
to
\begin{align*} \twonorm{\thetahat - \thetastar} &\leq \frac{1 +
\specnorm{\GamHat^{-1} \Gam - \Id_\dims}}{\sqrt{\numobs}} \;
  \Psi_\numobs(\IdMat_D; \delta),
\end{align*}
which is the statement of~\Cref{ThmLinFuncIV} (with slightly altered
notation).

\Cref{ThmLinFuncIV} makes distinct predictions when $\Umat$ is not the
identity.  Concretely, suppose that $\Umat \equiv \uvec \in \real^D$
is a vector, and introduce the shorthand
\begin{subequations}
\begin{align}
\label{EqnDefnLmat}  
  \Lmat & \defn \GamMat^{-1} \SigMat \GamMat^{-T}
\end{align}
for the relevant covariance matrix.  The leading term in
$\Psi_\numobs(\uvec; \delta)$ then can be bounded as \mbox{$\Exs
  |\uvec^T \GaussVec_\numobs| \leq \sqrt{\uvec^T \Lmat \uvec}$,}
whereas the leading order term in $\Psi_\numobs(\IdMat_D; \delta)$ can
be bounded as \mbox{$\Exs \|\GaussVec_\numobs\|_2 \leq \sqrt{
    \trace(\Lmat)}$.} Consequently, disregarding the other terms in
the result, we have an approximate bound of the form
\begin{align}
\label{EqnApproxBound}  
\sqrt{\numobs} \big| \inprod{\uvec}{\thetahat - \thetastar} \big| &
\precsim \sqrt{\uvec^T \Lmat \uvec} + \|\uvec^T(\GamHat^{-1} \Gam -
\Id_\dims)\|_2 \sqrt{\trace(\Lmat)}.
\end{align}
\end{subequations}
The leading term with $\uvec^T \Lmat \uvec$ is as we expect, given the
asymptotic behavior under classical scaling.  As for the second term,
since we expect that $\GamHat$ becomes close to $\Gam$ as $\numobs$
increases, the pre-factor in front of the second term will become
negligible for large $\numobs$.

However, in the non-asymptotic regime, when this term is still
non-negligible, the bound contains a quantity proportional to
$\sqrt{\trace(\Lmat)}$.  Note that this trace quantity will typically
be much larger---often by a dimension-dependent factor---than the
scalar $\sqrt{u^T \Lmat u}$.  Below we undertake some careful
numerical simulations to show that, at least in finite samples, the IV
error itself---and \emph{not just} our bound on it---also exhibits a
dependence on this trace term.


\subsubsection{Some numerical studies}
\label{SecHardEnsemble}

\begin{figure}[h!]
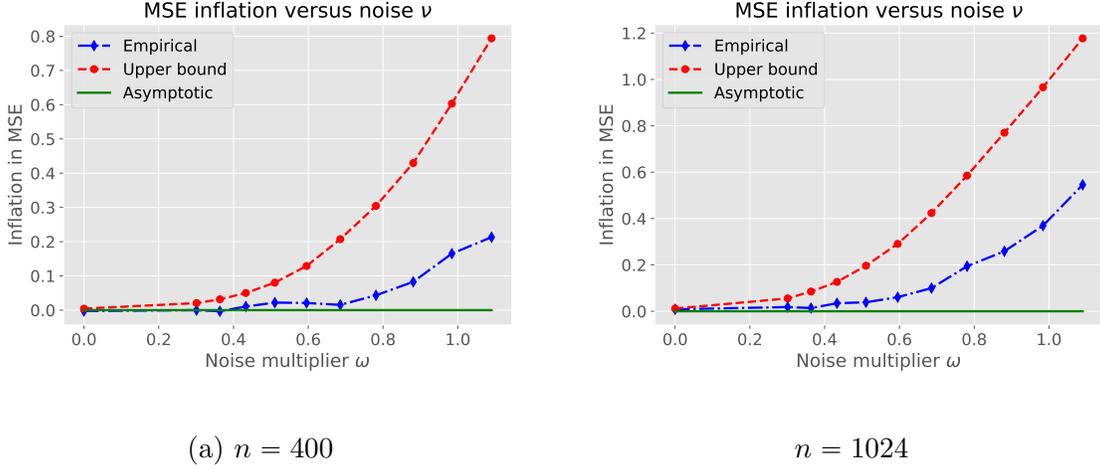

  \begin{center}
\begin{tabular}{ccc}
      \widgraph{0.45\textwidth}{\figdir/fig_hard_xnoise05_numtrials2500}
      &&
      \widgraph{0.45\textwidth}{\figdir/fig_hard_xnoise10_numtrials2500}
      \\ \\ (a) $\numobs = 400$ && $\numobs = 1024$
    \end{tabular}      
\caption{Plots of the log increase in mean-square error (MSE) versus
  the noise parameter $\omega$ for ensembles in dimension $\usedim$
  and $\numobs = 256$ with the error matrix $\Lmat \equiv
  \Lmat(\omega)$ in equation~\eqref{EqnDefnLmatOmega}.  (a)
  Simulations with $\plainstd = 0.50$.  (b) Simulations with
  $\plainstd = 1.0$.  In both panels, solid green lines correspond to
  the asymptotic log MSE (equal to zero by construction); dash-dot
  blue lines correspond to the IV empirical MSE estimated by $M =
  2500$ Monte Carlo trials; and dashed red lines correspond to the
  predictions from equation~\eqref{EqnApproxBound}.}
\label{FigXnoise}  
\end{center}
\end{figure}

Recalling the definition~\eqref{EqnDefnLmat} of the matrix $\Lmat$ and
following discussion, let us describe some numerical studies that
reveal how the IV estimate error itself depends on the trace term
$\trace(\Lmat)$.  In order to do so, we construct an ensemble of
problems, parameterized by a weight $\omega \geq 0$, for which
\begin{align}
\label{EqnDefnLmatOmega}
\Lmat & = \GamMat^{-1} \SigMat \GamMat^{-1} = \diag \big(1, \ldots,
\omega, \ldots, \omega \big).
\end{align}
See~\Cref{AppHardEnsemble} for the details of this construction.

Letting $\uvec = e_1$ be the standard basis vector with a single one
in position $1$, we then have $u^T \Lmat u = 1$ for any value of
$\omega$, whereas $\trace(\Lmat) = 1 + (\usedim - 1) \omega$.
Consequently, the leading term in our bound~\eqref{EqnApproxBound} is
independent of $\omega$, whereas the higher-order trace term grows
linearly in $\omega$.  In summary, then, the asymptotic prediction for
the $\numobs$-rescaled MSE is equal to $1$ for all values of $\omega$,
whereas our theory predicts that the higher-order terms should scale
with $\omega$.

In order to study the correspondence between our non-asymptotic
predictions and the IV error in practice, we simulated from this
ensemble in dimension $\usedim$ with sample size $\numobs = 256$,
instrument strength $\strinst = 1$, endogeneity level $\strendo =
0.10$, and the parameter $\omega$ ranging over the unit interval
$[0,1]$.  We did so for different levels of noise level $\plainstd \in
\{0.5, 1.0\}$.  For each setting of the parameters, we estimated the
MSE of the IV estimate useful based on $M = 2500$ Monte Carlo
trials. \Cref{FigXnoise} gives plots of the log increase in
mean-squared error (MSE) beyond the asymptotic prediction for both the
actual IV estimate (dash-dot blue lines), and our theoretical upper
bound (dashed red lines).  Panels (a) and (b) correspond to noise
levels $\plainstd = 0.50$ and $\plainstd = 1.0$, respectively.
Consistent with the theory, the empirical MSE increases with the noise
multiplier $\omega$ that parameterize the structure of the matrix
$\Lmat \equiv \Lmat(\omega)$.


\subsection{Computable confidence intervals}
\label{sec:comp-conf}

We now turn to the task of devising non-asymptotic and data-dependent
confidence intervals on the error $\betahat - \betastar$ in the IV
estimate.  In particular, we derive bounds that depend on the
empirical matrix $\GamHat = \frac{1}{\numobs} \sum_{i=1}^\numobs
\instrument_i \covariate_i^T$, along with a standard estimate of the
matrix \mbox{$\SigMat \defn \frac{1}{\numobs} \sum_{i=1}^\numobs
  \EE[\noise_i^2 \instrument_i \instrument_i^T]$}---namely, the random
matrix
\begin{subequations}
\begin{align}
\label{EqnDefnSigEst}  
\SigEst & \defn \frac{1}{\numobs-1} \sum_{i=1}^\numobs \noisehat^2_i
\instrument_i \instrument_i^T
\end{align}
where $\noisehat_i \defn \response_i -
\inprod{\covariate_i}{\betahat}$ are the residuals associated with the
IV estimate.  Combining these ingredients yields the standard
``sandwich'' estimate $\GamHat^{-1} \SigEst \GamHat^{-T}$ for the
asymptotic covariance.

In addition to our previous assumptions, our analysis requires the
following third-moment bound on the random vectors $\SigMat^{-1/2}
\noise_i \instrument_i$.
\begin{align}
\label{EqnThirdMoment}  
\frac{1}{\numobs} \sum_{i=1}^\numobs \EE \|\SigMat^{-1/2} (\noise_i
\instrument_i)\|_2^3 \leq m_3.
\end{align}
In addition, for a given error probability $\delta \in (0,1)$, we
state our results in terms of the quantile functional $\delta \mapsto
\quantile{\delta}$ defined by the relation
\begin{align}
\label{EqnQuantile}
\Prob \big[ |V| \geq \quantile{\delta} \big] = \delta \qquad
  \mbox{where $V \sim \mathcal{N}(0, 1)$.}
\end{align}
\end{subequations}
For example, when $\delta = 0.05$, we have $\quantile{\delta} \approx
1.96$.

\subsubsection{Bound for a general linear function}

With this set-up, we begin by stating a result that specifies a
confidence interval for a linear function of the IV error
\mbox{$\DelHat \defn \betahat - \betastar$.}  Defining the random
vectors $V_i \defn \inprod{v}{\instrument_i} \covariate_i \in
\real^\usedim$, our result involves an error term given by
\begin{align}
\label{EqnDefnKapErr}  
\errone_\numobs(\DelHat; \vvec) & \defn \sqrt{\DelHat^T \Qmat(\vvec)
  \DelHat} \qquad \mbox{where $\Qmat(\vvec) \defn \frac{1}{\numobs
    (\numobs-1)} \sum_{i < j} (V_i - V_j) (V_i - V_j)^T$.}
  \end{align}
\begin{theorem}
\label{ThmConfInt}
Suppose that, in addition to the assumptions of~\Cref{ThmStandardIV},
the third moment condition~\eqref{EqnThirdMoment} holds. Then for any
scalars $\pardelta, \pardelta' \in (0,1)$ and for any vector $\vvec
\in \real^d$, we have
  \begin{align}
\label{EqnConfInt}    
\sqrt{\numobs} |\inprod{\vvec}{\GamHat \DelHat}| & \leq
\quantile{\delta} \left \{\sqrt{\vvec^T \SigEst \vvec} +
\underbrace{\bound \|\vvec\|_2 \sqrt{\frac{8 \log(1/\pardelta')}{\numobs -
      1}}}_{\Term_1} + \underbrace{\errone_\numobs(\DelHat; \vvec)}_{\Term_2}
\right \}
\end{align}
with probability at least $1 - \pardelta - \pardelta' - \frac{c
  m_3}{\sqrt{\numobs}}$.
\end{theorem}
\noindent See~\Cref{AppProofThmConfInt} for the proof of this claim. \\

\noindent Let us discuss some features of this bound.
\vspace*{-0.1in}
\paragraph{Rough scaling:}
By inspection, we see that $\Term_1 = \mathcal{O}(1/\sqrt{\numobs})$
so it becomes negligible as $\numobs$ tends to infinity.  Note that we
have $\Term_2 \defn \errone_\numobs(\DelHat; \vvec) \leq
\sqrt{\specnorm{\Qmat(\vvec)}} \|\DelHat\|_2$.  Observe that
$\Qmat(\vvec)$ is a standard estimate for the covariance matrix of the
random vector $V = \inprod{v}{\instrument} \covariate$; consequently,
we expect it to be bounded with high probability under mild
conditions.  In this case, the quantity $\Term_2$ should be $o_p(1)$
whenever the IV estimate is consistent in probability (so that
$\|\DelHat\|_2$ goes to zero).  Thus, a compact summarization is given
by the bound $\sqrt{\numobs} \big|\inprod{\vvec}{\GamHat \DelHat}|
\leq \quantile{\delta} \big \{ \sqrt{\vvec^T \SigEst \vvec} + o_p(1)
\big \}$ with the claimed probability.

\paragraph{Bounds on  individual entries $|\DelHat_j|$:} Letting
$\uvec \in \real^\usedim$ be a vector, suppose that our goal is to
specify a confidence interval for the inner product
$\inprod{\uvec}{\DelHat}$.  As a concrete example, setting $\uvec =
e_j$, the standard basis vector with a one in position $j$, would
allow us to compute a confidence interval for entry $\betastar_j$ of
the unknown parameter vector $\betastar$.  In this case, we would like
to apply~\Cref{ThmConfInt} with the vector $\vhat \defn \GamHat^{-T}
e_j$, so that $\inprod{\vhat}{\GamHat \DelHat} = \DelHat_j$.  However,
this choice is \emph{not valid} since the vector $\vhat$ is random
(depending on $\GamHat$).

However, we can side-step this difficulty by
applying~\Cref{ThmConfInt} to the deterministic vector $\vvec \defn
\GamMat^{-T} \uvec$, and then bounding the difference between
$\GamHat$ and its expectation $\GamMat$.  As with our previous
non-asymptotic results, this difference can be captured via the
zero-mean difference matrix
\begin{align}
\Dmat & \defn \GamMat^{-1} \GamHat - \IdMat_\usedim.
\end{align}
We summarize in the following:
\begin{corollary}
\label{CorRefineConfInt}  
Under the conditions of~\Cref{ThmConfInt}, for any unit-norm vector
$\uvec \in \real^d$, we have
\begin{subequations}
  \begin{align}
\sqrt{\numobs} |\inprod{\uvec}{\DelHat}| & \leq \quantile{\delta}
\Biggr\{ \underbrace{\sqrt{\uvec^T \GamHat^{-1} \SigEst \GamHat^{-T}
    \uvec} + \bound \|\GamHat^{-T} \uvec \|_2 \sqrt{\frac{8
      \log(1/\pardelta')}{\numobs - 1}}}_{\Term_3} \Biggr \} +
\Term_4, \qquad \mbox{where} \\
\Term_4 & \defn \quantile{\delta} \specnorm{\Qmat(\vvec)}
\|\DelHat\|_2 + \specnorm{\Dmat} \Biggr \{ \quantile{\delta}
\specnorm{\GamHat^{-1} \SigEst^{1/2}} + \bou \specnorm{\GamHat}
\sqrt{\tfrac{8 \log(1/\pardelta')}{\numobs - 1}} + \sqrt{\numobs}
\|\DelHat\|_2 \Biggr \}
  \end{align}
\end{subequations}
with probability at least $1 - \pardelta - \pardelta' - \frac{c
  m_3}{\sqrt{\numobs}}$.
\end{corollary}

\noindent See~\Cref{AppProofCorRefineConfInt} for the proof. \\

Note that $\Term_3$ is a data-dependent quantity that can be
computed,\footnote{We assume here that the bound $\bou$ is known.}
whereas the additional error term $\Term_4$ involves $\DelHat$ and
$\Dmat$, and so cannot be computed.  However, we can argue that the
error term $\Term_4$ is lower order under relatively mild conditions.
The dominant component of $\Term_2$ is the quantity $\sqrt{\numobs}
\specnorm{\Dmat} \|\DelHat\|_2$.  In terms of its scaling with the
pair $(\numobs, \usedim)$:
\begin{itemize}
\item standard random matrix theory arguments (e.g., Chapter 5 in the
  book~\cite{dinosaur2019}) ensure that, under mild tail conditions on
  the pair $(\instrument, \covariate)$, we should have
  $\specnorm{\Dmat} \precsim \sqrt{\frac{\usedim}{\numobs}}$ with high
  probability.
\item from our results in the previous section, we typically have the
  bound $\|\DelHat\|_2 \precsim \sqrt{\frac{\usedim}{\numobs}}$ with
  high probability.
\end{itemize}
Putting together the pieces, we find the scaling
\begin{align*}
\Term_4 \asymp \sqrt{\numobs} \sqrt{\frac{\usedim}{\numobs}} \;
\sqrt{\frac{\usedim}{\numobs}} = \frac{\usedim}{\sqrt{\numobs}}.
\end{align*}
Consequently, it suffices to have $\usedim^2/\numobs \rightarrow 0$ in
order for $\Term_4$ to be lower order relative to $\Term_3$.


\subsubsection{Weakening dimension dependence}

\Cref{CorScalarConfInt} imposes a relatively stringent
$\frac{\usedim^2}{\numobs} \rightarrow 0$ condition to be valid. Here
we state a result that weakens this guarantee, albeit at the cost of
introducing larger pre-factors (i.e. $r_\pardelta$). Our result is
especially suited to the problem of uniformly controlling the
deviations $\sqrt{\numobs} |\langle u, \DelHat \rangle |$ uniformly
for all vectors $u$ in some finite set $\mathcal{U}$.  A standard
example would be $\mathcal{U} = \{e_1, \ldots, e_\usedim \}$,
corresponding to the set of standard basis vectors.

We begin by describing a bound that is valid for a single fixed vector
$u$.  Choose some $\GamErr$ such that
\begin{align*}
\twonorm{\GamHat^{-1}\uvec - \Gam^{-1}u} \leq \GamErr
\end{align*}
with probability exceeding $1 - \pardelta'$. Let $\CanBasis \defn
\{e_1, -e_1, \ldots, e_d, -e_d\}$ denote the set of canonical basis
vectors and their negations, and define
\begin{align*}
\Term_5 \defn \sup_{w \in \CanBasis} \twonorm{\Gam^{-1} u + \lambda
  w}, \qquad \text{and} \qquad \Term_6 \defn \sup_{w \in \CanBasis}
\sqrt{\specnorm{\Qmat(\Gam^{-1}u + \lambda w)}}.
\end{align*}
With this notation, we have the following guarantee:
\begin{corollary}
\label{cor:uniform-CI}
Under the conditions of~\Cref{ThmConfInt}, we have
\begin{align*}
\sqrt{\numobs} | \langle \uvec, \DelHat \rangle | \leq r_{\delta /
  2\usedim} \left\{ \sqrt{\uvec^T \GamHat^{-T} \SigEst
  \GamHat^{-1}\uvec} + \bound\Term_5 \cdot
\sqrt{\frac{8\log(2\usedim/\pardelta')}{\numobs-1}} + 2\GamErr
\specnorm{\SigEst} + \Term_6 \twonorm{\DelHat} \right\}
\end{align*}
with probability exceeding $1 - \pardelta - \pardelta' -
\frac{cm_3}{\sqrt{\numobs}}$.
\end{corollary}

\noindent See~\Cref{AppProofCorUnifInt} for the proof. \\

We claim that in typical settings, the the leading order term in this
bound will be the quantity $\sqrt{\uvec^T \GamHat^{-T} \SigEst
  \GamHat^{-1}\uvec}$ as long as $\frac{\usedim
  \log(\usedim)}{\numobs} \rightarrow 0$.  In the strong instrument
setting and typical tail conditions, standard concentration arguments
ensure that $\lambda = O(\sqrt{\frac{\usedim}{\numobs}})$; moreover,
since $\specnorm{\Qmat}$ is a bounded random variable, we have
$\Term_6 = O(\sqrt{\log \usedim})$.  Since $\twonorm{\DelHat} =
O(\sqrt{\usedim}{\numobs})$, the claim follows as long as
$\frac{\usedim \log(\usedim)}{\numobs} \rightarrow 0$.  Thus, we have
obtained a more favorable dependence on $d$ than our earlier result.

However, this relaxation comes at the cost of $r_{\pardelta/2\usedim}$
instead of $r_{\pardelta}$, and straightforward calculations show that
$r_{\pardelta/2d} \asymp \log(\usedim/\pardelta)$. Although this log
factor is undesirable, it is unavoidable in situations when we want to
have uniform confidence intervals over multiple different $u_i$'s. We
can apply a union bound over~\Cref{cor:uniform-CI} to construct a
uniform confidence region for $\sqrt{\numobs}| \langle u_i, \DelHat
\rangle |$ for $i = 1, \ldots, k$ and the constant pre-factor would be
$r_{\pardelta/(2dk)} \asymp \log(\usedim k / \pardelta)$ instead. As
an example, if we wanted uniform confidence interval over all entries
(in which case $u_i = e_i$ for $i = 1, \ldots, d$), then the
pre-factor becomes
\begin{align*}
r_{\delta/\usedim^2} \asymp \log(\usedim^2/\pardelta) \asymp \log(\usedim/\pardelta).
\end{align*}


\subsubsection{Specialization to univariate setting}

We now specialize to the univariate ($\usedim = 1$) setting, in which
case we can give more precise results that exhibit interesting
behavior in the \emph{weak instrument} regime discussed at the end
of~\Cref{SecStandIV}.  In particular, for weak instruments, the error
term $\Term_2$ in~\Cref{ThmConfInt} need not be negligible.
Incorporating its effects leads to finite-sample corrections to
standard CIs. At an intuitive level, our guarantee in~\Cref{CorScalarConfInt}
is valid under the distributional assumption that
\begin{align*}
\frac{1}{\sqrt{\numobs}} \sum_{i=1}^\numobs \instrument_i \noise_i \approx \mathcal{N}(0, \sigma^2),
\end{align*}
i.e.~that the distribution of $\frac{1}{\sqrt{\numobs}} \sum_{i=1}^\numobs \instrument_i \noise_i$ is approximately normal with some variance. Unlike the standard weak instrument literature,
it makes no assumptions on the distribution of $\frac{1}{\numobs}\sum_{i=1}^\numobs \instrument_i \covariate_i$.

We begin by stating a precise refinement of~\Cref{ThmConfInt} for a
univariate problem.  It involves the following coefficient:
\begin{align}
\label{EqnDefnScalarKap}  
\kappa_\numobs &= \frac{1}{\sqrt{\numobs}} \; \frac{
  \sqrt{\tfrac{1}{\numobs(\numobs-1)} \sum_{i < j} (\instrument_i
    \covariate_i - \instrument_j \covariate_j)^2}}{
  \big|\frac{1}{\numobs} \sum_{i=1}^\numobs \instrument_i
  \covariate_i|} \; = \; \frac{1}{\sqrt{\numobs}}
\frac{|\Qmat(1)|}{|\GamHat|},
\end{align}
where we recall the definitions of the matrices $\GamHat$ and $\Qmat$
from equations~\eqref{EqnDefnGamma} and~\eqref{EqnDefnKapErr},
respectively. $\Qmat(1)$ is precisely the estimate of the covariance
of the sample $\{\instrument_i \covariate_i\}_{i=1}^\numobs$.

\begin{corollary}  
  \label{CorScalarConfInt}
  Under the assumptions of~\Cref{ThmConfInt}, in the univariate case
  $\usedim = 1$, the following statements hold, each with probability
  at least $1 - \pardelta - \pardelta' - \frac{c
    m_3}{\sqrt{\numobs}}$.
\begin{enumerate}
\item[(a)] If $\kappa_\numobs \quantile{\delta} < 1$, then
\begin{subequations}
  \begin{align}
\label{EqnCorrectedA}    
\sqrt{\numobs} \big| \paramhat - \paramstar \big| & \leq
\frac{\quantile{\delta}}{1 - \quantile{\delta} \mkap} \left( \sqrt{ \,
  \GamHat^{-1} \SigEst \GamHat^{-1}} + \frac{\bound}{\big|\GamHat
  \big|} \cdot \sqrt{\frac{8\log(1/\pardelta')}{\numobs - 1}}\right).
\end{align}
\item[(b)] If $\mkap \quantile{1-\delta} > 1$, then
  \begin{align}
\label{EqnCorrectedB}    
\sqrt{\numobs} \big | \paramhat - \paramstar \big| & \leq
\frac{\quantile{1-\delta}}{\kappa_\numobs \quantile{1-\delta} - 1}
\left( \sqrt{ \, \GamHat^{-1} \SigEst \GamHat^{-1}} +
\frac{\bound}{\big|\GamHat \big|} \cdot
\sqrt{\frac{8\log(1/\pardelta')}{\numobs - 1}} \right).
\end{align}
\end{subequations}
\end{enumerate}
\end{corollary}
\noindent See~\Cref{SecProofCorScalarConfInt} for the proof of this
corollary. \\

Let us make a few comments to interpret this result, beginning with
part (a).  The leading term $\quantile{\delta} \sqrt{ \, \GamHat^{-1}
  \SigEst \GamHat^{-1}}$ defines the boundary of the typical estimate
of a $1-\delta$ confidence interval.  (Recall that for a typical level
of $\delta = 0.05$, corresponding to a $95 \%$ CI, we have
$\quantile{\delta} \approx 1.96$.)  Up to the correction factor
$\frac{1}{1 - \quantile{\delta} \mkap}$ (and the higher order term),
part (a) guarantees that this interval used in practice is accurate.
From its definition~\eqref{EqnDefnScalarKap}, in the classical regime,
we expect that $\mkap = O_p(1/\sqrt{\numobs})$; this scaling follows
since both the numerator $\Qmat(1)$ and denominator $|\GamHat|$
converge to order one quantities.  Consequently, as long as $\mkap$ is
sufficiently small, part (a) certifies that the classical confidence
interval is valid.  Note that $\mkap$ can be computed from the data,
so that this validity can be verified by the user.

\Cref{CorScalarConfInt} also provides guidance in the regime of
\emph{weak instruments}, in which case $\mkap$ is a random coefficient
of order one.  Concretely, consider ensembles of problems based on
$\numobs$ i.i.d. samples from the generative model~\eqref{EqnEndo} for
the covariate $\covariate$, level of endogeneity $\strendo = 1$, and
additional noise level $\plainstd = 0$.  We obtain weak instruments by
setting $\strinst = \strinst_1/\sqrt{\numobs}$ for some $\strinst_1 >
0$, where the instrument is chosen as $\instrument \in \{-1, +1 \}$
equiprobably.  With these choices, we have
\begin{align}
\label{EqnMyWeak}  
  \GamMat & = \Exs[\covariate \instrument] =
  \frac{\strinst}{\sqrt{\numobs}} \Exs[\instrument^2] =
  \frac{\strinst}{\sqrt{\numobs}},
\end{align}
so that we are in the weak instrument regime.  Let us compute the form
of $\mkap$ for ensembles with parameters $\plainstd = 0$ and $\strendo
= 1$ in equation~\eqref{EqnEndo}.  Noting that $\instrument_i
\covariate_i = \frac{\strinst_1}{\sqrt{\numobs}} + \instrument_i
\noise_i$, a little calculation yields
\begin{align*}
\mkap & = \frac{\sqrt{\frac{1}{\numobs ( \numobs-1)} \sum_{i < j}
    (\instrument_i \noise_i - \instrument_j \noise_j)^2}}{ \big|
  \strinst_1 + V_\numobs \big|} \qquad \mbox{where $V_\numobs \defn
  \frac{1}{\sqrt{\numobs}} \sum_{i=1}^\numobs \instrument_i
  \noise_i$.}
\end{align*}
For large $\numobs$, the numerator converges to $\var(\instrument
\noise) = 1$, whereas $V_\numobs \rightsquigarrow N(0, 1)$, so that
$\mkap$ behaves like $1/|U(\strinst_1)|$, where $U(\strinst_1) \sim N(
\strinst_1, 1)$.  In this regime, the correction factor $(1 -
\quantile{\delta} \mkap)^{-1}$ in~\Cref{CorScalarConfInt}(a) can play
a significant role.  We explore this phenomenon further
in~\Cref{SecKappaSims}.

Finally, observe that in the regime of large $\mkap$, part (b)
guarantees that we can substantially \emph{shrink} the classical
confidence interval.  In~\Cref{SecKappaSims}, we also construct an
ensemble---involving a dependent sequence of covariates
$\{\covariate_i \}_{i=1}^\numobs$, as permitted by the generality of
our theory---for which $\mkap$ is large with high probability.  In
this regime, we can produce CIs that are substantially smaller than
the classical prediction.


\section{Applied uses of confidence intervals}
\label{SecApplications}

In this section, we explore the applied uses of our confidence
intervals (CIs).  We begin in~\Cref{SecKappaSims} with a numerical
study of the corrected CIs from~\Cref{CorScalarConfInt} on simulated
data; in~\Cref{sec:empirical}, we turn to an application of IV methods
to study the effect of air pollution (PM2.5 levels) on various health
outcomes, using an original dataset compiled by the authors.


\subsection{Numerical study of corrected confidence intervals}
\label{SecKappaSims}

In this section, we undertake a numerical study of the corrected CIs
given in~\Cref{CorScalarConfInt}.  Of particular to interest is to
explore the effect of $\mkap$, as defined in
equation~\eqref{EqnDefnScalarKap}, that determines the correction
applied to the classical CIs based only on the sandwich estimator.
All experiments were run with $\delta = 0.05$, corresponding to
a CI with nominal coverage of $95\%$.

\paragraph{Corrected CIs for small $\mkap$:}

We begin by exploring the effect of the correction factor
$\frac{\quantile{\delta}}{1 - \quantile{\delta} \mkap}$ applied in
part (a) of~\Cref{CorScalarConfInt}.  
\begin{figure}[ht!]
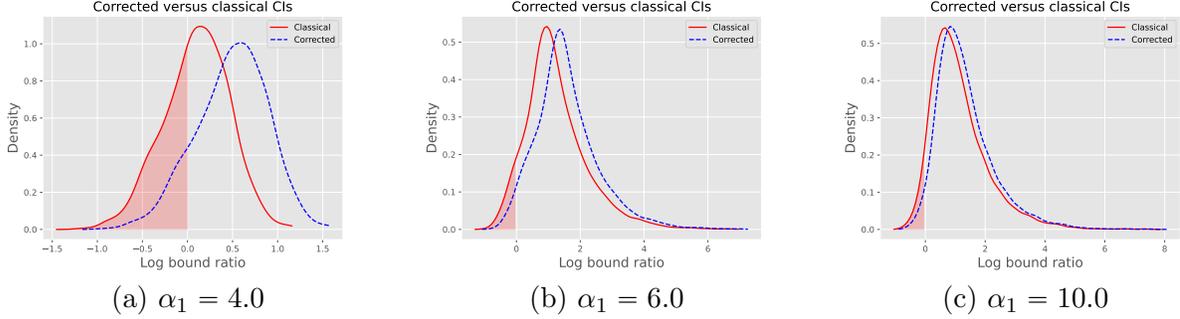

  \begin{center}
    \begin{tabular}{ccc}
      \widgraph{0.33\textwidth}{\figdir/fig_kappa_weak4_n256_numtrials10000_final}
      &
      \widgraph{0.33\textwidth}{\figdir/fig_kappa_medium_n256_numtrials10000_final}
      &
      \widgraph{0.33\textwidth}{\figdir/fig_kappa_strong_n256_numtrials10000_final}
      \\
      (a) $\strinst_1 = 4.0$ & (b) $\strinst_1 = 6.0$ & (c)
      $\strinst_1 = 10.0$
    \end{tabular}
    \caption{Study of the bounds
      in~\Cref{CorScalarConfInt}(a). Experiments with sample size
      $\numobs = 256$ with weak instrument scaling depending on the
      instrument strength $\strinst \in \{4.0, 6.0, 10.0 \}$ in panels
      (a), (b) and (c), respectively.  Density estimates of the log
      ratio between the classical CI estimate and the empirical IV
      error (red solid line) and the corrected CIs and the empirical
      IV error (blue dashed line).}
    \label{FigCorrectedCI}
  \end{center}
\end{figure}
With sample size $\numobs = 256$ and noise parameters $(\strendo,
\plainstd) = (1, 0)$, we generated IV problems with weak instruments
using the procedure described in equation~\eqref{EqnMyWeak}.  We
considered three different instrument strengths---namely, $\strinst =
\strinst_1/\sqrt{\numobs}$ for $\strinst_1 \in \{4.0, 6.0, 10.0\}$.
For each setting of the parameters, we performed $M = 10,000$ Monte
Carlo trials, computing the log ratio of either the classical CI to
the true IV error, or the corrected CI from
equation~\eqref{EqnCorrectedA} to the true IV error.  The classical
CIs exhibited under coverage for instrument strengths $\strinst \in
\{4.0, 6.0 \}$, with actual estimated coverages at $92\%$ and $93 \%$,
respectively.  We computed the corrected CIs only when they were
applicable according to the criteria~\eqref{EqnCorrectedA}; the bounds
were applicable in 19\%, 87\% and 100\% of trials for the settings
$\strinst_1 \in \{4.0, 6.0, 10.0 \}$, respectively.

\Cref{FigCorrectedCI} shows density plots of these log ratios (with
densities estimated using a standard Gaussian KDE).  The classical CI
ratios are plotted in red solid lines, with the corrected ones plotted
in blue dashed lines; the shaded red area corresponds to the fraction
of Monte Carlo trials for which the classical CIs were violated when
the bound applied.

\paragraph{Corrected CIs for large $\mkap$:}

Part (b) of~\Cref{CorScalarConfInt} shows that the classical CIs can
be overly large in problems for which the coefficient $\mkap$ is quite
large, so that the condition $\mkap \quantile{1-\delta} > 1$ applies.
(With our setting $\delta = 0.05$, we have $\quantile{1-\delta} =
0.06$, so that we need $\mkap \gg 1$.)  From the
definition~\eqref{EqnDefnScalarKap} of $\mkap$, it can be seen that
large values of $\mkap$ are not likely to occur when the pairs
$(\instrument, \covariate)$ are independent and identically
distributed.  However, recall that our theory imposes \emph{only}
distributional assumptions on the sequence $\{\instrument_i \noise_i
\}_{i=1}^\numobs$.  There are \emph{no} distributional assumptions on
the covariates $\{\covariate_i \}_{i=1}^\numobs$, and this freedom can
be exploited so as to construct ensembles for which $\mkap$ is large
with high probability.

\begin{figure}[h!]
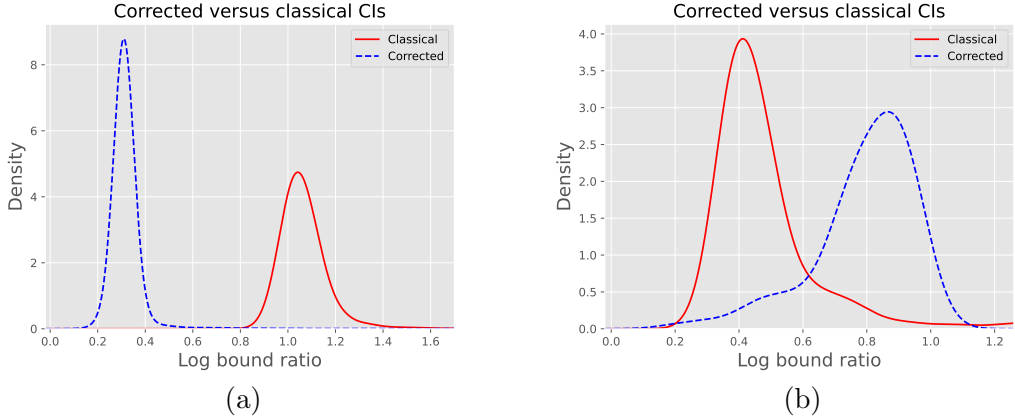

  \begin{center}
    \begin{tabular}{cc}
      \widgraph{0.45\textwidth}{\figdir/fig_badkap_weak_n256_numtrials10000_final}
      &
      \widgraph{0.45\textwidth}{\figdir/fig_badkap_strong_n256_numtrials10000_final}
      \\
      (a) & (b)
    \end{tabular}
    \caption{ Density estimates of the log ratio between the classical
      CI estimate and the empirical IV error (red solid line) and the
      $\mkap$-corrected CIs and the empirical IV error (blue dashed
      line), computed over $M = 10000$ trials.  (a) Instrument
      strength $\strinst_1 = 0.25$.  (b) Instrument strength
      $\strinst_1 = 0.75$.}
    \label{FigBadKap}
  \end{center}
\end{figure}

In particular, suppose that we generate covariates $X$ from the
ensemble~\eqref{EqnEndo} with endogeneity $\strendo = 0.10$ and
additional noise level $\plainstd = 10.0$, where the additional noise
vector $W$ is constructed as follows:
\begin{align*}
  W_i = \instrument_i G_i \qquad \mbox{for $i = 1, \ldots, \numobs$}
\end{align*}
where $G \in \real^\numobs$ is a zero-mean Gaussian random vector such
that $\sum_{i=1}^\numobs G_i = 0$ with probability one.\footnote{ In
particular, the vector $G$ has entries with variance $\var(G_i) = 1$
and \mbox{$\cov(G_i, G_j) = -1/(\numobs-1)$} for $i \neq j$.} Since
$Z_i \in \{-1, +1 \}$, this construction ensures that
$\sum_{i=1}^\numobs Z_i W_i = \sum_{i=1}^\numobs G_i = 0$, so that the
addition of $W$ to $\covariate$ has no effect on $\sum_{i=1}^\numobs
Z_i X_i$.  On the other hand, adding $W$ does affect the numerator
$\sum_{i < j} (Z_i X_i - Z_j X_j)^2$ in the definition of $\mkap$, in
particular making it larger.

We conducted experiments using this ensemble with sample size
$\numobs$, and two settings of the instrument strength $\strinst =
\strinst_1/\sqrt{\numobs}$---namely, $\strinst_1 \in \{0.25, 0.75\}$.
As before, we plotted the log ratio between the CI limits, either the
classical one or the corrected using part (b)
of~\Cref{CorScalarConfInt}.  \Cref{FigBadKap} shows density estimates
of these log ratios based on $M = 10,000$ trials.  respectively.  With
the weaker instrument $\strinst_1 = 0.25$ shown in panel (a), the
inverse of the correction factor $\frac{\quantile{1-\delta}}{\mkap
  \quantile{1-\delta} - 1}$ can be quite large, often around a factor
of $10$ (or $1$ on the log scale).  This effect dissipates for the
somewhat stronger instrument level $\strinst_1 = 0.75$ shown in panel
(b), and for very strong instruments, it is highly unlikely for
$\mkap$ to be particularly large.


\subsection{IV study of PM2.5 exposure}
\label{sec:empirical}

The conclusions drawn from this section are based on an original
dataset compiled by the authors. The data is organized into census
tracts, i.e. each census tract counts as a particular observation
within our dataset. The motivating factor to consider data at a
community level, as opposed to a individual level (the standard
approach in similar studies), is due to the various publicly
accessible datasets that provide community-level measurements. We are
interested in estimating the effect of \pollm~exposure on various
health metrics of the community. The response $\response$ is taken
from the PLACES project~\cite{placesdata} which consists of prevalence
rates of various negative health conditions including arthritis,
asthma, blood pressure, cancer, heart disease, and stroke. measured
between 2019-2020. We report our estimates of the effect of pollution
for a variety of these health conditions. Here, the endogenous
covariate of interest $\covariate$ is the average daily \pollm~levels
for each tract measured over the course of 2016, taken from the Public
Health Tracking Network~\cite{pm25data}. The exogeneous covariates
$\exogcov$ include various other measurement for the census tract,
such as race and demographics, educational attainment, employment by
sector, median income, poverty rates, health insurance rates,
indicators for states, and more, collected primarily from the 2020 US
Census~\cite{censusdata}. Our data set is summarized in
Table~\ref{table:summary}. We provide a visualization of the average
\pollm levels in the continental United States in 2016 at the
resolution of individual counties in Figure~\ref{fig:pollution}.

\begin{table*}
\caption{Summary statistics by census tracts}
\label{table:summary}
\begin{center}
\renewcommand{\arraystretch}{1.15}
\begin{tabular}{ l @{\qquad \qquad \qquad} c @{\qquad \qquad} c @{\quad \qquad} c }
\hline \hline & Average & SD & USA (2020) \\ \hline Population & 4002
& 1629 & \\ Median household income & \$ 68,600 & \$34,600 & \$74.580
\\ Male & 49\% & 5\% & 50\% \\ Non-white & 30\% & 26\% & 28\% \\ Under
18 & 22\% & 7\% & 22\% \\ Over 67 & 15\% & 8\% & 17\% \\ Poverty rate
& 14\% & 11\% & 11\% \\ Associates degree or higher & 39\% & 19\% &
35\% \\ Uninsurance rate & 9\% & 7\% & 11\% \\ \hline \hline
\end{tabular}
\subcaption*{Our dataset consists of 58,761 census tracts treated as
  individual units. The first two reflects the average and standard
  deviation of each of the measurements indicated above across all the
  census tracts. The third column represents the actual average across
  the United States, to indicate that our data is a representative
  sample of the United States. Since the US Census data is often
  reported as number of people within a census tract that falls into a
  specific group, we have decided to use the fraction/percentage to
  standardize the tracts. }
\end{center}
\caption{Instrumental variable estimate of effects of pollution on
  negative health outcomes incidence rates and their standard
  errors. \\}
\label{table:iv-results}
\begin{center}
\begin{tabular}{l @{ \qquad}  @{\qquad} c @{\qquad}  @{\qquad} c @{\qquad} }
\hline \hline & OLS & IV \\ 
\hline \multirow{2}{*}{Arthritis} & -1.9 & 12.2 \\ & (0.0) & (1.3) \\ 
[7pt] \multirow{2}{*}{Asthma} & -0.5 & 5.2 \\ & (0.0) & (0.5) \\ 
[7pt] \multirow{2}{*}{Cancer} & 0.0 & 1.3 \\ & (0.0) & (0.3) \\ 
[7pt] \multirow{2}{*}{Heart disease} & -0.2 & 2.9 \\ & (0.0) & (0.4) \\ 
[7pt] \multirow{2}{*}{High blood pressure} & -0.8 & 1.7 \\ & (0.1) & (1.8) \\ 
[7pt] \multirow{2}{*}{Poor mental health} & -0.2 & 5.6 \\ & (0.0) & (0.7) \\ 
[7pt] \multirow{2}{*}{Poor sleep} & 0.5 & 0.1 \\ & (0.1) & (1.2) \\ [7pt]
\multirow{2}{*}{Stroke} & -0.1 & 1.5 \\ & (0.0) & (0.2) \\ \hline
\hline
\end{tabular}
\subcaption*{Results are reported in \% increase per 10 \pollm}
\end{center}
\end{table*}

\begin{figure}[t]
\begin{center}
\includegraphics[scale=0.35]{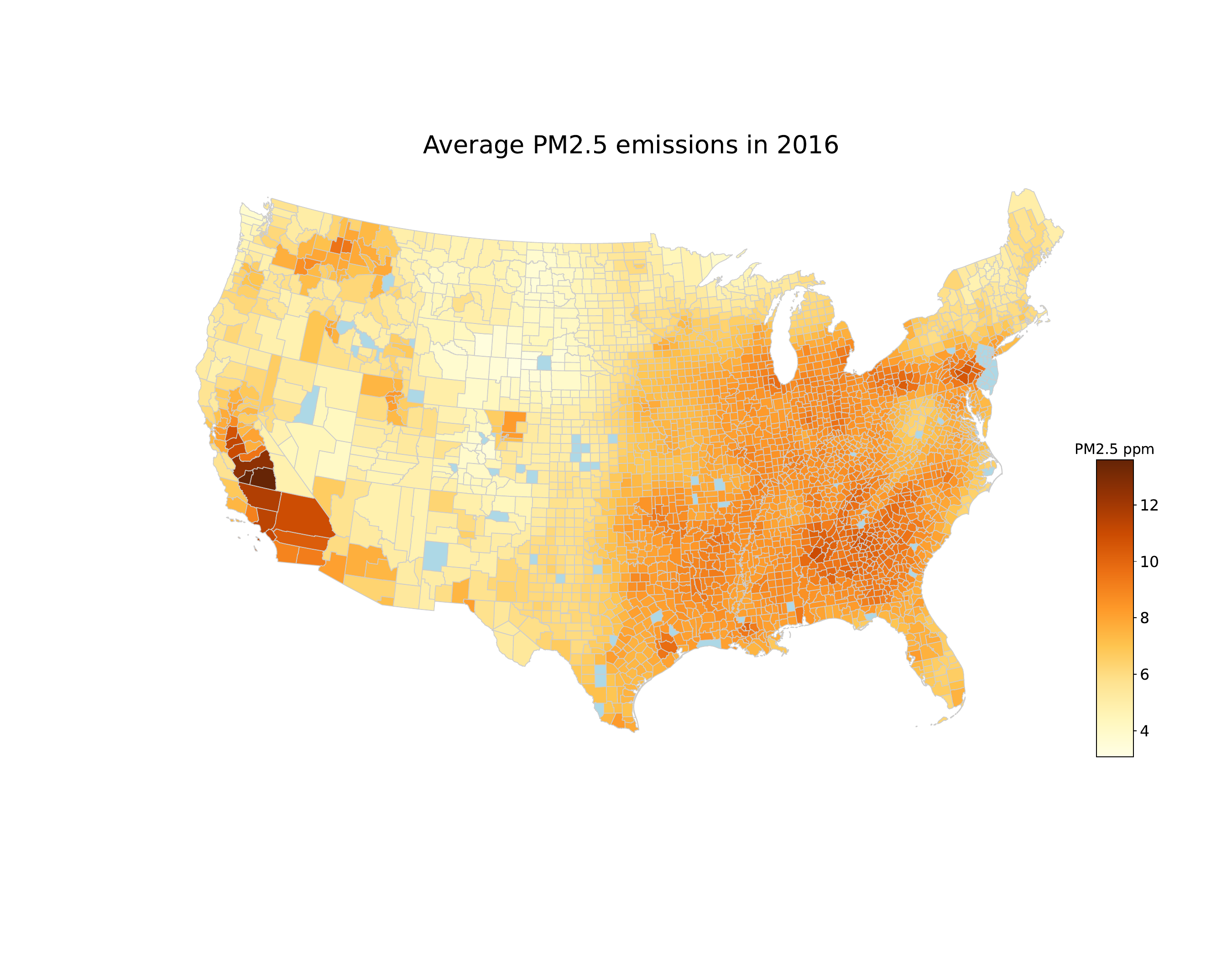}
\caption{Average \pollm~levels in the continental USA in 2016; the
darker the color the higher the fraction. Light blue indicates
missing data for those specific counties.}
\label{fig:pollution}
\end{center}
\end{figure}

The instrument $\instrument$ is a measure of the exposure of a given
census tract to pollution caused by wildfires. Abstractly, our
instrument is a gauge of a community's susceptibility to wildfire
smoke. We argue that this satisfies the exclusion restriction:
proximity to wildfires has well-documented effects on the day-to-day
pollution a specific community experiences, but is independent of
other causes of negative health outcomes. An individuals exposure to
wildfire pollution is dictated by geography and climate which, in the
modern era, has little effect on one's living and health conditions
(especially the ones of interest). Thus it serves as a natural
experiment by which we can analyze the effects of pollution on health
outcomes. To describe our specific instrument, let $\mathcal{D}$
denote a dataset consisting of all the wildfires in 2016 that burned
for at least 100 acres, as reported in the data given by the National
Interagency Fire Center~\cite{firedata}. We define our instrument as
\begin{align}
\label{eqn:inst-defn}
\instrument_i = \mathbf{1}\{ \instrument_i^* \geq c \}, \qquad
\text{where} \qquad \instrument^*_i = \frac{1}{\card{\mathcal{D}}}
\sum_{f \in \mathcal{D}} \frac{f_{\text{size}}}{(f_{\text{distance to
      $i$}})^2},
\end{align}
i.e. for every census tract $i$, we average across all the wildfires
$f \in \mathcal{D}$ in 2016 the ratio of the size of the fire over the
squared distance of the tract to the fire. Our final instrument then
is a truncated version of the above instrument, $\instrument_i =
\mathbf{1}\{\instrument_i^* \geq c\}$ for some constant $c$ chosen
such that a reasonable fraction of the total instruments are both $0$
and $1$. Figure~\ref{fig:wildfire} depicts the fraction of census
tracts within each county such that the wildfire smoke index was high,
i.e. $Z_i = 1$. With this instrument, our first stage $F$-statistic is
$256$, so our instrument is strong. In the appendix, we report
robustness checks using different thresholds and formulas and most of
the results remain consistent with those reported below.

\begin{figure}[t]
\begin{center}
\includegraphics[scale=0.35]{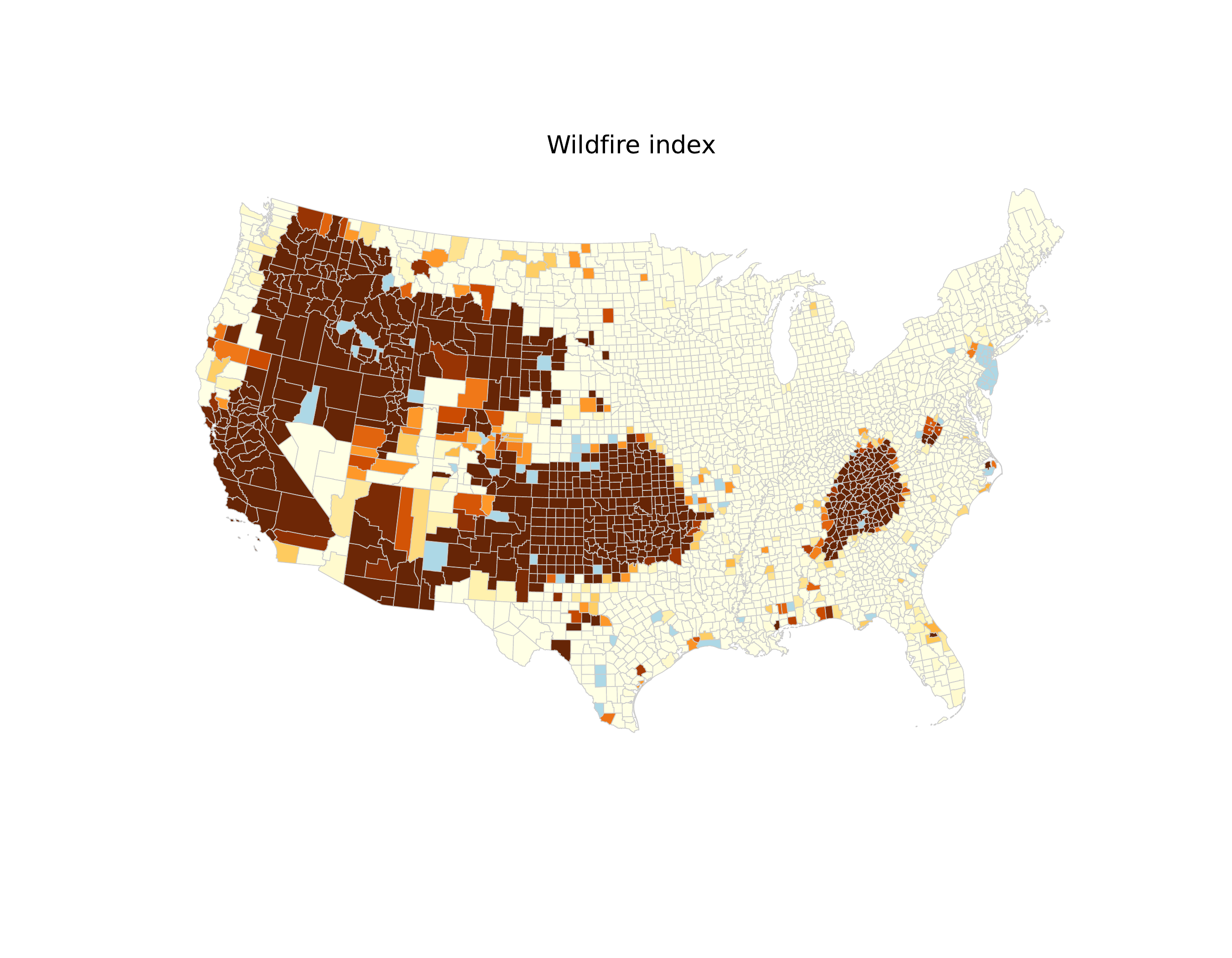}
\caption{Fraction of census tracts within each county such that
  \mbox{$Z_i = 1$;} the darker the color the higher the
  fraction. Light blue indicates missing data for those specific
  counties.}
\label{fig:wildfire}
\end{center}
\end{figure}

Our results are presented in Table~\ref{table:iv-results}. The
standard errors are robust standard errors, i.e.~computed via the
standard errors as proved in~\Cref{CorScalarConfInt}.  We remark that
because our units are geographic in nature, the standard errors may be
inaccurate due to spatial correlations that violate the independence
assumption required by our estimator. The reported numbers are
measured in \% per 10 \pollm, i.e. the increase in incidence rates due
to an increase in \pollm~concentration by 10. For reference, we also
included the results of running OLS on the same dataset. Note that in
every instance, OLS returns negative effect, indicating that an
increase in pollution decreases the rates of these negative health
outcomes. The instrumental variables regression indicates that
\pollm~pollution `causes' increases in the prevalence of these health
conditions. For example, an increase in an average \pollm~level of
$10$ in a given community results in a 3.3\% increase in heart disease
rates. Much of these results confirm existing association studies in
the medical literature, i.e. the relationship between pollution and
arthritis~\cite{adami2021association}, asthma~\cite{tiotiu2020impact},
cancer~\cite{turner2020outdoor}, heart disease~\cite{miller2020air},
and stroke~\cite{shah2015short}. We remark that we also included the
estimates of the effects of \pollm~on high blood pressure and poor
sleep to indicate that our choice of instrument does not return
significant effects for every possible response.


\section{Discussion}
\label{sec:discussion}

In this paper, we presented a non-asymptotic analysis of the classical
linear instrumental variable estimator in the
$\ell_2$-norm. Additionally, we established guarantees for estimating
the parameter of interest in the presence of exogeneous covariates. To
the best of our knowledge, these are the first instances of such
guarantees. We also establish, in the one-dimensional setting,
confidence intervals for the IV estimate, as well as a novel measure
for the strength of an instrument. In the strong-instrument setting
our confidence intervals matches the classical asymptotic ones up to
higher-order terms, but our measure of strength can be used to
appropriately adjust the asymptotic ones to account for potentially
weak instruments. Additionally, if the instrument is very weak our
guarantees can be used to substantially sharpen the asymptotic
confidence intervals.  Furthermore, we applied our results to analyze
the effect of air pollution on various negative health outcomes, using
wildfire exposure as an instrument. Our IV estimates indicates that
increased exposure to \pollm~pollution can increase ones risk of
various health conditions such as arthritis, asthma, cancer, heart
disease, poor mental health, and strokes.

Our results still leave open avenues of further investigation. The
non-asymptotic confidence intervals we have derived still tend to be
quite conservative or have stringent conditions that need to be met in
order to be used. To make them truly usable in practice there are
substantial improvements that would need to be made. Additionally, we
have not addressed the question of confidence intervals in the
presence of exogeneous covariates, which is the default setting in
practice. In this paper we also only focused on the just-identified
setting; when the number of instruments matches the number of
covariates. Often times in practice we may have more instruments than
we do covariates, known as the over-identified setting.  Finally,
instrumental variables are a special case of a broad class of
estimators known as generalized method of moments (GMMs). Establishing
non-asymptotic results for any of these more general settings is an
interesting direction for future work.


\bibliography{bibliography}
\bibliographystyle{amsalpha}


\appendix

\section{Proof of~\Cref{ThmStandardIV,ThmLinFuncIV}}
\label{SecProofNonAsymp}

In this appendix, we prove our two theorems that provide
non-asymptotic bounds on the IV estimate itself (\Cref{ThmStandardIV})
and linear mappings thereof (\Cref{ThmLinFuncIV}).  As noted following
the statement of~\Cref{ThmLinFuncIV}, it is actually more general
than~\Cref{ThmStandardIV}, so that it suffices to prove it.

\subsection{Main argument}
\label{sec:proof-i.n.i.d.-bounded}

Introduce the shorthand $W_\numobs \defn \empsum{i} \instrument_i
\noise_i$.  Using the form of the standard IV estimate, we can write
\begin{align*}
  \betahat - \betastar & = \GamHat^{-1} W_\numobs \; = \; \GamMat^{-1}
  W_\numobs + \Big( \GamHat^{-1} \GamMat - \IdMat_\dims \Big)
  \GamMat^{-1} W_\numobs.
\end{align*}
Applying $\Umat^T \in \real^{k \times \usedim}$ to both sides yields
\begin{align*}
\Umat^T (\betahat - \betastar) & = \; \Umat^T \GamMat^{-1} W_\numobs +
\Big[ \Umat^T \big( \GamHat^{-1} \GamMat - \IdMat_\dims \big) \Big]
\GamMat^{-1} W_\numobs.
\end{align*}
Consequently, by the triangle inequality, we have
\begin{align}
  \label{EqnUsefulBound}
\|\Umat^T(\betahat - \betastar)\|_2 & \leq \underbrace{\Big
  \|\empsum{i} \Umat^T \GamMat^{-1} \instrument_i \noise_i \Big
  \|_2}_{\Term_1} + \gamn{\GamMat; \Umat} \: \; \underbrace{\Big
  \|\empsum{i} \GamMat^{-1} \instrument_i \noise_i \Big
  \|_2}_{\Term_2}
\end{align}
  where $\gamn{\GamMat; \Umat} \defn \specnorm{\Umat^T( \GamHat^{-1}
    \GamMat - \IdMat_\dims)}$.

From this decomposition, we see that the central problem is
controlling the Euclidean norms of two sums of random variables.  In
order to do so, we make use of the following auxiliary result:

\begin{lemma}
\label{LemGD}
Let $\{ \randvec_i \}_{i=1}^\numobs$ be an independent sequence of
zero-mean random vectors such that $\twonorm{\randvec_i} \leq \Bou$
almost surely.  Then for any $\delta \in (0, 1)$, we have
\begin{align}
\label{EqnGDBound}  
\twonorm{\sum_{i=1}^\numobs \randvec_i} \leq
\EE\left[\twonorm{\sum_{i=1}^\numobs \randvec_i} \right] +
\sqrt{2
  \log(\tfrac{1}{\pardelta}) \EE \specnorm{\sum_{i=1}^\numobs
    \randvec_i \randvec_i^T}} + 3 \Bou \log(\tfrac{1}{\pardelta})
\end{align}
with probability at least $1 - \pardelta$.
\end{lemma}
\noindent See~\Cref{AppGD} for the proof. \\

\noindent We use this lemma to bound each of the two terms on the RHS
of equation~\eqref{EqnUsefulBound}.

\paragraph{Bounding the quantity $\Term_1$:}
We first apply~\Cref{LemGD} with $\randvec_i \defn \Umat^T \Gam^{-1}
(\noise_i \instrument_i)$ in order to bound the first term $\Term_1$.
From our set-up, each $\randvec_i$ is zero-mean and satisfies the
$\bou$-boundedness condition; moreover, the sequence is independent by
assumption.  Recalling the vector $\GaussVec_\numobs$ from
equation~\eqref{EqnDefnGaussVec}, we have $\sum_{i=1}^\numobs
\randvec_i = \sqrt{\numobs} \Umat^T \GaussVec_\numobs$.  As for the
term $\specnorm{\sum_{i=1}^\numobs \randvec_i \randvec_i^T}$, we have
\begin{align*}
  \sum_{i=1}^\numobs \randvec_i \randvec_i^T & = \Umat^T \GamMat^{-1}
  \Big( \frac{1}{\numobs} \sum_{i=1}^\numobs \noise_i^2 \instrument_i
  \instrument_i^T \Big) \GamMat^{-T} \Umat \; = \; \numobs \Umat^T
  \GamMat^{-1} \SigSpec \GamMat^{-T} \Umat,
\end{align*}
where the final equality uses the definition of $\SigSpec$ from
equation~\eqref{EqnDefnSigMat}.  Dividing by $\numobs$ and putting
together the pieces, we have shown that
\begin{subequations}
\begin{align}
\label{EqnTermOne}
  \Big \|\empsum{i} \Umat^T \GamMat^{-1} \instrument_i \noise_i \Big
  \|_2 & \leq \frac{\E \|\Umat^T
    \GaussVec_\numobs\|_2}{\sqrt{\numobs}} + \frac{1}{\sqrt{\numobs}}
  \sqrt{2 \log(\tfrac{1}{\pardelta}) \EE \specnorm{\Umat^T
      \GamHat^{-1} \SigSpec \GamHat^{-T} \Umat}} + \frac{3 \Bou
    \log(\tfrac{1}{\pardelta})}{\numobs}.
\end{align}


\paragraph{Bounding the quantity $\Term_2$:}  The previous argument
applies to a general $\Umat$ so that we can apply it with $\Umat =
\IdMat_\usedim$.  (Our assumptions imply that the random vectors
$\GamMat^{-1} \instrument_i \noise_i$ are also $\bou$-bounded.)
Consequently, by recourse to the bound~\eqref{EqnTermOne}, we can
conclude that
\begin{align}
\label{EqnTermTwo}
\Term_2 & \leq \frac{\E \|\GaussVec_\numobs\|_2}{\sqrt{\numobs}} +
\frac{1}{\sqrt{\numobs}} \sqrt{2 \log(\tfrac{1}{\pardelta}) \EE
  \specnorm{\GamHat^{-1} \SigSpec \GamHat^{-T}}} + \frac{3 \Bou
  \log(\tfrac{1}{\pardelta})}{\numobs}.
\end{align}
\end{subequations}

\vspace*{0.1in}

Finally, combining the two bounds~\eqref{EqnTermOne}
and~\eqref{EqnTermTwo} with our initial
decomposition~\eqref{EqnUsefulBound} yields the claim
of~\Cref{ThmLinFuncIV}.


\subsection{Proof of~\Cref{LemGD}}
\label{AppGD}

We control $\twonorm{\sum_{i=1}^\numobs \randvec_i}$ by representing
it as the supremum of an empirical process.  Note that we can write
\begin{align*}
\twonorm{\sum_{i=1}^\numobs \randvec_i} = \sup_{\|u\|_2 = 1}
\sum_{i=1}^\numobs \inprod{u}{\randvec_i} = \sup_{f_u \in \fclass}
\sum_{i=1}^\numobs f_u(\randvec_i),
\end{align*}
where we define $\fclass = \{f_u: f_u(x) = \langle u, x \rangle, \, u
\in \mathbb{S}^{\dims - 1} \}$, i.e.~$\fclass$ is the set of linear
functions with unit norm coefficient vectors.  With this set-up, we
perform the necessary calculations in order to apply a result due to
Klein and Rio (see~\Cref{thm:klein-rio} in~\Cref{AppEmpProcess}).

For all $f_u \in \fclass$, we have
\begin{align*}
\EE[f_u(\randvec_i)] = 0, \quad \text{and} \quad \|f_u\|_\infty =
\sup_{z} |f_u(V)| = |\langle u, V \rangle| \leq \Bou \quad a.s.,
\end{align*}
where we have used the assumption that the vectors $V = \noise
\Gam^{-1} Z$ are $\Bou$-uniformly bounded.

Additionally, we have
\begin{align*}
\nu \defn \EE\left[ \sup_{f_u \in \fclass} \sum_{i=1}^\numobs
  f^2(\randvec_i) \right] = \EE\left[ \sup_{\|u\|_2 = 1} \, u^T \big(
  \sum_{i=1}^\numobs \randvec_i \randvec_i^T \big) u \right] & = \EE
\specnorm{\sum_{i=1}^\numobs \randvec_i \randvec_i^T} \\
& = \Exs \bigspecnorm{ \Gam^{-1} \SigSpec \Gam^{-T}},
\end{align*}
where the final equality uses the definition of $V_i$ and
$\SigSpec$.  With these calculations in place,
applying~\Cref{thm:klein-rio} from~\Cref{AppEmpProcess} yields
\begin{align*}
\twonorm{\frac{1}{\numobs} \sum_{i=1}^\numobs \randvec_i} \leq
\EE\twonorm{\frac{1}{\numobs} \sum_{i=1}^\numobs \randvec_i} +
\frac{1}{\sqrt{\numobs}} \sqrt{2\log(\tfrac{1}{\pardelta})
  \EE\twonorm{\Gam^{-1} \SigSpec \Gam^{-T}}} + \frac{3 \Bou
  \log(\tfrac{1}{\pardelta})}{\numobs},
\end{align*}
with probability at least $1 - \pardelta$.

\subsection{Extensions to unbounded random variables}

In our current statements of~\Cref{ThmStandardIV,ThmLinFuncIV}, we
impose $\bou$-boundedness conditions on the random vectors
$\GamMat^{-1} \instrument_i \noise_i$.  These conditions allow us to
apply a concentration inequality for empirical processes due to
Bousquet with sharp constants.  However, the underlying proof
technique is not restricted to bounded random vectors; by leveraging
results on unbounded empirical processes, we can obtain
generalizations of~\Cref{ThmStandardIV,ThmLinFuncIV}.  For example,
Adamczak~\cite{Ada08} provides bounds on the suprema of empirical
processes under Orlicz norm conditions.  One can also make use of
truncation arguments to handle this more general setting.


\section{Proof for confidence interval guarantees}

In this section, we provide the proofs of our two main results on
confidence intervals---namely, ~\Cref{ThmConfInt}
and~\Cref{CorScalarConfInt}.


\subsection{Proof of~\Cref{ThmConfInt}}
\label{AppProofThmConfInt}

Let $V \sim \mathcal{N}(0, 1)$ be standard normal.  By the
Berry--Esseen theorem (c.f.~\Cref{thm:berry}), for any $r > 0$, we
have
\begin{align}
\label{EqnGeneralBerry}  
\left| \PP \left[ \sqrt{\numobs} \big|u^T \GamHat \DeltaHat \big| \geq
  r \sqrt{u^T \SigMat u} \right] - \PP \big[|V | \geq r \big] \right|
\leq \frac{c m_3}{\sqrt{\numobs}}.
\end{align}
The choice $r = \quantile{\delta}$ ensures that $\PP [|V| \geq
  \quantile{\delta}] = \delta$, whence
\begin{subequations}
\begin{align}
\label{EqnInterOne}
\sqrt{\numobs} \big| u^T \GamHat \DeltaHat \big| \leq
\quantile{\delta} \sqrt{u^T \SigMat u}
\end{align}
with probability at least $1 - \delta - \frac{c m_3}{\sqrt{\numobs}}$.

Introduce the shorthand $\shortw_i \defn
\inprod{\uvec}{\instrument_i}$, and observe that $\sqrt{u^T \SigMat u}
= \sqrt{\frac{1}{\numobs} \sum_{i=1}^\numobs \EE[\noise_i^2
    \shortw_i^2]}$.  Using the $\bou$-boundedness and independence
condition, we can apply Theorem 10 in the
paper~\cite{maurerpontil2009} so as to guarantee that
\begin{align}
\label{EqnInterTwo}  
\sqrt{u^T \SigMat u} & \leq \sqrt{\frac{1}{\numobs(\numobs-1)} \sum_{i
    < j} (\noise_i \shortw_i - \noise_j \shortw_j)^2} + \bound
\sqrt{\frac{8\log(1/\pardelta')}{\numobs - 1}}.
\end{align}
with probability at least $1 - \pardelta'$.

Recall the definition~\eqref{EqnDefnKapErr} of
$\errone_\numobs(\DelHat)$, as well as the matrix $\SigEst$ from
equation~\eqref{EqnDefnSigEst}.  We claim that the proof will be
complete if we can show that
\begin{align}
\label{EqnIntermediate}  
\sqrt{\frac{1}{\numobs(\numobs-1)} \sum_{i < j} (\noise_i \shortw_i -
  \noise_j \shortw_j)^2} & \leq \sqrt{u^T \SigEst u} +
\errone_\numobs(\DelHat)
\end{align}
\end{subequations}
Indeed, if this bound holds, then combining it with
inequalities~\eqref{EqnInterOne} and~\eqref{EqnInterTwo} yields
\begin{align}
\label{EqnAppendixCI}  
\sqrt{\numobs} \big| u^T \GamHat \DeltaHat \big| \leq
\quantile{\delta} \Big \{ \sqrt{u^T \SigEst u} +  \bound
\sqrt{\frac{8\log(1/\pardelta')}{\numobs - 1}} +
\errone_\numobs(\DelHat) \Big \}
\end{align}
with probability at least $1 - \pardelta - \pardelta' - \frac{c
  m_3}{\sqrt{\numobs}}$, which establishes the
bound~\eqref{EqnConfInt} from~\Cref{ThmConfInt}.

\paragraph{Proof of the bound~\eqref{EqnIntermediate}:}
We begin by observing the decomposition
\begin{align}
\label{EqnAdecomp}  
  \underbrace{\noise_i \shortw_i - \noise_j \shortw_j}_{a_{ij}} &
  \quad = \quad \underbrace{\noiseest_i \shortw_i - \noiseest_j
    \shortw_j}_{b_{ij}} \quad + \quad \underbrace{
    \inprod{\covariate_i}{\DelHat} \shortw_i -
    \inprod{\covariate_j}{\DelHat} \shortw_j}_{c_{ij}}.
\end{align}
Since $\|a\|_2 \leq \|b\|_2 + \|c\|_2$ by the triangle inequality, we
have
\begin{align*}
\underbrace{\sqrt{\sum_{i < j} (\noise_i \shortw_i - \noise_j
    \shortw_j)^2}}_{\|a\|_2} & \leq \underbrace{
  \sqrt{\sum_{i < j}
    \left(\noiseest_i \shortw_i - \noiseest_j \shortw_j
    \right)^2}}_{\|b\|_2} + \underbrace{ \sqrt{ \sum_{i < j}
    \left(\inprod{\covariate_i}{\DelHat} \shortw_i -
    \inprod{\covariate_j}{\DelHat} \shortw_j \right)^2}}_{\|c\|_2}.
\end{align*}
Dividing both sides by $\sqrt{\numobs (\numobs-1)}$ and recalling the
definition~\eqref{EqnDefnKapErr} of $\errone_\numobs(\DelHat)$, we
have shown that
\begin{align*}
\sqrt{\frac{1}{\numobs(\numobs-1)} \sum_{i < j} (\noise_i \shortw_i -
  \noise_j \shortw_j)^2} & \leq \sqrt{\frac{1}{\numobs (\numobs-1)}
  \sum_{i < j} \left(\noiseest_i \shortw_i - \noiseest_j \shortw_j
  \right)^2} + \errone_\numobs(\DelHat)
\end{align*}

By~\Cref{lems:sample-var}(a), we have the equivalence
\begin{align*}
 \frac{1}{\numobs (\numobs-1)} \sum_{i < j} \left(\noiseest_i
 \shortw_i - \noiseest_j \shortw_j \right)^2 & = \frac{1}{\numobs-1}
 \sum_{i=1}^\numobs \big( \noisehat_i \shortw_i - \frac{1}{\numobs}
 \sum_{j=1}^\numobs \noisehat_j \shortw_j \big)^2
\end{align*}
But recalling that $\shortw_i = \inprod{u}{\instrument_i}$, we have
\begin{align*}
  \sum_{j=1}^\numobs \noisehat_j \shortw_j \; = \;
  \inprod{u}{\sum_{j=1}^\numobs \big(y_i -
    \inprod{\covariate_i}{\betahat} \big) \instrument_j} &
  \stackrel{(\star)}{=} 0,
\end{align*}
where equality $(\star)$ follows from the optimality conditions that
define the IV estimate.  Thus, we have shown that
\begin{align*}
 \frac{1}{\numobs (\numobs-1)} \sum_{i < j} \left(\noiseest_i
 \shortw_i - \noiseest_j \shortw_j \right)^2 & = \frac{1}{\numobs-1}
 \sum_{i=1}^\numobs \noisehat^2_i \shortw_i^2 \; = \; u^T \SigEst u,
\end{align*}
thereby proving the claim~\eqref{EqnIntermediate}.

\subsection{Proof of~\Cref{CorRefineConfInt}}
\label{AppProofCorRefineConfInt}

We begin by applying~\Cref{ThmConfInt} with the vector $\vvec =
\GamMat^{-1} \uvec$, from which we are guaranteed to have
\begin{subequations}
  \begin{align}
    \label{EqnRefineOne}
\sqrt{\numobs} |\inprod{\uvec}{\GamMat^{-1} \GamHat \DelHat}| & \leq
\quantile{\delta} \left \{ \sqrt{\uvec^T \GamMat^{-1} \SigEst
  \GamMat^{-T} \uvec} + \bound \|\GamMat^{-T} \uvec\|_2 \sqrt{\frac{8
    \log(1/\pardelta')}{\numobs - 1}} + \errone_\numobs(\DelHat)
\right \}.
\end{align}
with the claimed probability.  An application of the triangle
inequality ensures that
\begin{align}
  \sqrt{\numobs} |\inprod{\uvec}{\DelHat}| & \leq \sqrt{\numobs}
  |\inprod{\uvec}{\GamMat^{-1} \GamHat \DelHat}| + \sqrt{\numobs}
  |\inprod{\uvec}{(\GamMat^{-1} \GamHat - \IdMat_\usedim)\DelHat}|
  \nonumber \\
\label{EqnRefineTwo}
  & = \sqrt{\numobs} |\inprod{\uvec}{\GamMat^{-1} \GamHat \DelHat}|
+ \sqrt{\numobs} \big|u^T \Dmat \DelHat \big|.
\end{align}
Similarly, we have
\begin{align}
\sqrt{\uvec^T \GamMat^{-1} \SigEst \GamMat^{-T} \uvec} \: = \: \|
\SigEst^{1/2} \GamMat^{-T} \uvec\|_2 & \leq \| \SigEst^{1/2}
\GamHat^{-T} \uvec\|_2 + \| \SigEst^{1/2} \big(\GamMat^{-T} -
\GamHat^{-T} \big) \uvec \|_2 \nonumber \\
\label{EqnRefineThree}
& = \sqrt{\uvec^T \GamHat^{-1} \SigEst \GamHat^{-T} \uvec} + \|
\SigEst^{1/2} \GamHat^{-T} \Dmat^T \uvec \|_2.
\end{align}
\end{subequations}
Combining the bounds~\eqref{EqnRefineTwo} and~\eqref{EqnRefineThree}
with our original inequality~\eqref{EqnRefineOne} yields 
  \begin{align*}
    \sqrt{\numobs} |\inprod{\uvec}{\GamMat^{-1} \GamHat \DelHat}| &
    \leq \quantile{\delta} \left \{ \sqrt{\uvec^T \GamHat^{-1} \SigEst
      \GamHat^{-T} \uvec} + \bound \|\GamMat^{-T} \uvec\|_2
    \sqrt{\frac{8 \log(1/\pardelta')}{\numobs - 1}} \right \} +
    \quantile{\delta} \errone_\numobs(\DelHat) +
    \errtwo_\numobs(\DelHat, \Dmat),
  \end{align*}
where
\begin{align*}
  \errtwo_\numobs(\DelHat, \Dmat) & \defn
  \underbrace{\quantile{\delta} \| \SigEst^{1/2} \GamHat^{-T} \Dmat^T
    \uvec \|_2 + \bound \; \|\GamHat^T \Dmat^T \uvec \|_2 \;
    \sqrt{\tfrac{8 \log(1/\pardelta')}{\numobs - 1}}}_{\Term_4} +
  \underbrace{\sqrt{\numobs} |\uvec^T\Dmat \DelHat|}_{\Term_5}.
\end{align*}
As noted previously, we have $\errone_\numobs(\DelHat) \leq
\specnorm{\Qmat(\vvec)} \|\DelHat\|_2$, and moreover, we see that
\begin{align*}
\Term_4 & \leq \specnorm{\Dmat} \left \{ \quantile{\delta}
\specnorm{\GamHat^{-1} \SigHat^{1/2}} + \Bou \specnorm{\GamHat}
\sqrt{\tfrac{8 \log(1/\pardelta')}{\numobs - 1}} \right \} \quad
\mbox{and} \quad \Term_5 \leq \sqrt{\numobs} \specnorm{\Dmat}
\|\DelHat\|_2.
\end{align*}
Putting together the pieces completes the proof.      


\subsection{Proof of~\Cref{cor:uniform-CI}}
\label{AppProofCorUnifInt}

The main idea is to control the error over the $\ell_1$-ball centered at $\Gam^{-1} \uvec$, and instead of uniformly controlling over some $\epsilon$-net, we can uniformly control over the vertices of the $\ell_1$-ball due to the fact that the optimal solution of a linear program is at the extremal points. 

For convenience, we use $\BB_1(\vvec; \lambda)$ to denote the $\ell_1$ ball around $\vvec$ with radius $\lambda$, i.e.
\begin{align*}
\BB_1(v; \lambda) = \{ w \in \RR^\usedim \, : \, \| w- v \|_1 \leq \lambda \} =  \text{conv} \left\{ w \pm \lambda e_i \, : \, i \in [\usedim] \right\}.
\end{align*}
We have that $\GamHat^{-1}\uvec \in \BB_1(\vvec;\lambda)$ and thus
\begin{align*}
\sqrt{\numobs} \, | \langle \uvec, \DelHat \rangle | &= \sqrt{\numobs} \, | \langle \GamHat^{-1} u, \GamHat \DelHat \rangle | \\
&\stackrel{(a)}{\leq} \sqrt{\numobs} \, \sup_{w \in \BB_1(\Gam^{-1}\uvec; \lambda)} | \langle w, \GamHat \DelHat \rangle | \\
&\stackrel{(b)}{=} \sqrt{\numobs} \cdot \max_{\substack{v \in \{+\Gam^{-1}\uvec, -\Gam^{-1}\uvec\} \\ w \in \CanBasis}} \, \langle \uvec + \lambda w, \GamHat \DelHat \rangle \\
&= \sqrt{\numobs} \cdot \max_{w \in \CanBasis} |\langle \Gam^{-1} \uvec + \lambda w, \GamHat \DelHat \rangle |
\end{align*}
Step (a) follows from above, and step (b) follows from the fact that the optimal solution for any linear program exists at one of the vertices.

By Berry-Esseen we have
\begin{align*}
&\PP\left( \sqrt{\numobs} \cdot \max_{w \in \CanBasis} |\langle \Gam^{-1} \uvec + \lambda w, \GamHat \DelHat \rangle | \geq \epsilon \right) \\
& \qquad \qquad \leq \PP \left(\max_{w \in \CanBasis} |\langle \Gam^{-1} \uvec + \lambda w, \SigMat^{1/2} V \rangle | \geq \epsilon \right) + \frac{cm_3}{\sqrt{\numobs}}
\end{align*}
where $V \sim \mathcal{N}(0, \mathbf{I}_d)$. By a union bound we have
\begin{align*}
\PP \left(\max_{w \in \CanBasis} |\langle \Gam^{-1} \uvec + \lambda w, \SigMat^{1/2} V \rangle | \geq \epsilon \right) &\leq  \sum_{i=1}^d \PP\left( | \langle \Gam^{-1} \uvec + \lambda e_i, \SigMat^{1/2} V \rangle | \geq \epsilon \right) 
\\& \qquad \qquad + \sum_{i=1}^d \PP\left( | \langle \Gam^{-1} \uvec - \lambda e_i, \SigMat^{1/2} V \rangle | \geq \epsilon \right).
\end{align*}
Using the shorthand 
\begin{align*}
\sigma^2(e_i) \defn (\Gam^{-1}\uvec + \lambda e_i)^T \SigMat (\Gam^{-1}\uvec + \lambda e_i)
\end{align*}
we have $ \langle \Gam^{-1} \uvec - \lambda e_i, \SigMat^{1/2} V \rangle \sim \mathcal{N}(0, \sigma^2(e_i))$, and that
\begin{align*}
\PP\left( | \langle \Gam^{-1} \uvec + \lambda e_i, \SigMat^{1/2} V \rangle | \geq \epsilon \right) = 2 - 2 \Phi\left( \frac{\epsilon}{\sigma(e_i)} \right). 
\end{align*}
If we use $\epsilon = r_{\pardelta/2d} \sup_{w \in \CanBasis} \sigma(w)$, we obtain
\begin{align*}
\max_{w \in \CanBasis} |\langle \Gam^{-1} \uvec + \lambda w, \SigMat^{1/2} V \rangle | \leq r_{\delta/2d} \sup_{w \in \CanBasis} \sigma(w)
\end{align*}
 with probability exceeding $1 - \pardelta$.

 We now focus on controlling the final term. We have that by applying Maurer and Pontil uniformly,
\begin{align*}
\sup_{w \in \CanBasis} \sigma(w) &\leq \sup_{w \in \CanBasis} \underbrace{\sqrt{(\Gam^{-1}u + \lambda w)^T \SigEst (\Gam^{-1}\uvec + \lambda w)}}_{ \STerm_1} \\
&\qquad \qquad + b \sqrt{\frac{8\log(2d/\pardelta')}{\numobs-1}} \cdot \sup_{w \in \CanBasis} \|\Gam^{-1}\uvec + \lambda w \|_2 \\
&\qquad \qquad + \sup_{w \in  \CanBasis} \underbrace{e_n(\DelHat; \Gam^{-1}\uvec + \lambda w)}_{\STerm_2}.
\end{align*}
We analyze the terms individually; for $\STerm_1$ we have
\begin{align*}
\STerm_1 = \| \Gam^{-1}\uvec + \lambda w \|_{\SigEst} &\leq \| \GamHat^{-1}\uvec \|_{\SigEst} + \| \Gam^{-1}u - \GamHat^{-1}\uvec \|_{\SigEst} + \lambda \| w \|_{\SigEst} \\
&\leq \sqrt{u^T \GamHat^{-T} \SigEst \GamHat^{-1} \uvec} + 2 \lambda \cdot  ||| \SigEst |||_2.
\end{align*}
For $\STerm_2$ we have
\begin{align*}
\STerm_2 \leq \| \DelHat\|_2 \cdot \sup_{w \in \CanBasis} \sqrt{|||\mat{Q}_n(\Gam^{-1}u + \lambda w) |||_2} 
\end{align*}

Putting together the pieces, we conclude
\begin{align*}
\sqrt{\numobs} |\langle u, \DelHat \rangle | &\leq r_{\pardelta/2d} \left\{ \sqrt{u^T \GamHat^{-T}  \SigEst \GamHat^{-1} u} + b \cdot \sqrt{\frac{8\log(2d/\pardelta')}{\numobs-1}} \cdot \sup_{w \in \CanBasis} \| \Gam^{-1} u + \lambda w \|_2  \right. \\ 
& \qquad \qquad \qquad \qquad  \left. + 2\lambda  \cdot \specnorm{\SigEst} + \| \DelHat \|_2 \cdot \sup_{w \in \CanBasis} \sqrt{\specnorm{\mat{Q}_n(\Gam^{-1}u + \lambda w)}}  \right\},
\end{align*}
as desired.


\subsection{Proof of~\Cref{CorScalarConfInt}}
\label{SecProofCorScalarConfInt} 

We now prove the corollary applicable to the univariate case ($\usedim
= 1$), so that both $\GamHat$ and $\SigEst$ are scalar quantities, and
$\uvec = 1$.

\paragraph{Proof of part (a):}

Beginning with the definition~\eqref{EqnDefnKapErr} of
$\errone_\numobs(\DelHat)$, in the scalar case, it simplifies to
\begin{align*}
  \errone_\numobs(\DelHat) & = \frac{1}{\sqrt{\numobs}}
  \frac{\sqrt{\frac{1}{\numobs (\numobs-1)} \sum_{i < j}
      \big(\instrument_i \covariate_i - \instrument_j
      \covariate_j)^2}}{|\frac{1}{\numobs} \sum_{i=1}^\numobs
    \instrument_i \covariate_i|} \; |\DelHat| \; = \; \sqrt{\numobs}
  \; \mkap |\GamHat| \; |\DelHat|,
\end{align*}
where we recall the definition~\eqref{EqnDefnScalarKap} of the scalar
$\mkap$.

Substituting this relation into the
bound~\eqref{EqnAppendixCI} and re-arranging yields
\begin{align*}
\sqrt{\numobs} |\GamHat| |\DeltaHat| \; \big \{ 1 - \quantile{\delta}
\mkap \big \} & \leq \quantile{\delta} \Big \{ \sqrt{\SigEst} + \bound
\sqrt{\frac{8\log(1/\pardelta')}{\numobs - 1}}  \Big \}.
\end{align*}
As long as $1 - \quantile{\delta} \mkap > 0$, we can rescale both
sides by $(1 - \quantile{\delta} \mkap) \; |\GamHat|$, which yields
the claimed bound.

\paragraph{Proof of part (b):}

Returning to our Berry--Esseen bound~\eqref{EqnGeneralBerry}, we set
$r = \quantile{1-\delta}$ to conclude that
\begin{subequations}
\begin{align}
\label{EqnInterOneB}
\sqrt{\numobs} |\GamHat| \; |\DeltaHat| & \geq \quantile{1 - \delta}
\sqrt{\SigMat}
\end{align}
with probability at least $1 - \delta - \frac{c m_3}{\sqrt{\numobs}}$.
(Here we recall that $u = 1$ in the univariate case, and that both
$\GamHat$ and $\SigMat$ are scalars.)

From Theorem 10 in the paper~\cite{maurerpontil2009}, we are
guaranteed to have
\begin{align}
\label{EqnInterTwoB}  
\sqrt{\SigMat} & \geq \sqrt{\frac{1}{\numobs(\numobs-1)} \sum_{i < j}
  (\noise_i \shortw_i - \noise_j \shortw_j)^2}
- \bound
\sqrt{\frac{8\log(1/\pardelta')}{\numobs - 1}}
\end{align}
with probability at least $1 - \pardelta'$, where in this scalar case
$\shortw_i = \inprod{u}{\instrument_i} = \instrument_i$.

Again using the decomposition~\eqref{EqnAdecomp}, the triangle
inequality implies that
\begin{align}
  \label{EqnInterThreeB}
\underbrace{\sqrt{\sum_{i < j} (\noise_i \shortw_i - \noise_j
    \shortw_j)^2}}_{\|a\|_2} & \geq \underbrace{ \sqrt{ \sum_{i < j}
    \left(\inprod{\covariate_i}{\DelHat} \shortw_i -
    \inprod{\covariate_j}{\DelHat} \shortw_j \right)^2}}_{\|c\|_2} -
\underbrace{ \sqrt{\sum_{i < j} \left(\noiseest_i \shortw_i -
    \noiseest_j \shortw_j \right)^2}}_{\|b\|_2}.
\end{align}
\end{subequations}
Combining equations~\eqref{EqnInterOneB}, ~\eqref{EqnInterTwoB}
and~\eqref{EqnInterThreeB}, we have shown that
\begin{align*}
  \sqrt{\numobs} \: |\GamHat| \: |\DelHat| & \geq \quantile{1 - \delta}
  \Big \{ \sqrt{\numobs} \mkap |\GamHat| |\DelHat|
  - \sqrt{\SigEst} -
  \bound \sqrt{\frac{8\log(1/\pardelta')}{\numobs - 1}} \Big \}.
\end{align*}
Re-arranging terms yields
\begin{align*}
  \sqrt{\numobs} \: |\GamHat| \: |\DelHat| \big \{ \mkap
  \quantile{1-\delta} - 1 \big \} & \leq \quantile{1 - \delta} \Big \{
  \sqrt{\SigEst} + \bound \sqrt{\frac{8\log(1/\pardelta')}{\numobs -
      1}} \Big \}.
\end{align*}
As long as $\mkap \quantile{1-\delta} > 1$, we can divide both sides
by $|\GamHat| \big( \mkap \quantile{1-\delta} - 1 \big)$ to conclude
the proof.


\section{Some auxiliary results}

In this appendix, we collect the statements of various tail bounds,
Berry--Esseen bounds, and sample variance bounds used in our analysis.

\subsection{Tail and concentration inequalities}
\label{AppEmpProcess}

Portions of our analysis make use of Bernstein's inequality:
\begin{theorem}
\label{thm:bernstein}
Let $\{X_i\}_{i=1}^\numobs$ be independent zero-mean random variables
such that $|X_i| \leq b$ almost surely for each $i$.  Then
\begin{align}
\PP\left(\sum_{i=1}^\numobs X_i \geq t \right) \leq \exp\left(-
\frac{\tfrac{1}{2} t^2}{\sum_{i=1}^\numobs \EE[X_i^2] + \tfrac{1}{3}
  bt }\right) \qquad \mbox{for each $t > 0$.}
\end{align}
\end{theorem}

There are various extensions of Bernstein's inequality to suprema of
empirical processes.  In particular, let $\fclass$ be a countable
class of functions $f:\mathcal{X} \to \RR$, and let $\{ X_i
\}_{i=1}^\numobs$ be a sequence of independent random variables in
$\mathcal{X}$ such that $\Exs[f(X_i)] = 0$ for each $f \in \fclass$.
It is frequently of interest to control the tails of the random
variable \mbox{$Z_\numobs \defn \sup_{f \in \fclass}
  \sum_{i=1}^\numobs f(X_i)$.}  In order to do so, we make use of
following result due to Klein and Rio~\cite{klein2005empproc}:
\begin{theorem}
\label{thm:klein-rio}
Suppose that $\sup_{f \in \fclass, i \in [\numobs]} |f(X_i)| \leq
\bound$ almost surely.  Then for all $t > 0$, we have
\begin{align}
\label{EqnKleinRio}  
\PP[Z_\numobs \geq \EE[Z_\numobs] + t] \leq \exp \left( - \frac{t^2}{2
  \nu^2 + 3 \bound t} \right)
\end{align}
where $\nu^2 \defn \EE[\sup \limits_{f \in \fclass} \sum_{i=1}^\numobs
  f^2(X_i)]$.
\end{theorem}

\subsection{Berry--Esseen theorem}

We also make use of some quantitative versions of the Berry--Esseen
theorem.
\begin{theorem}[\cite{raic2019berry}]
\label{thm:berry}
Let $\{X_i\}_{i=1}^\numobs$ be a sequence of zero-mean independent
random vectors in $\RR^d$ such that $\sum_{i=1}^\numobs \EE[X_i X_i^T]
= \mathbf{I}_d$, and define $W = \sum_{i=1}^\numobs X_i$.  Letting $V
\sim \mathcal{N}(0, \mathbf{I}_d)$ be standard normal, we have
\begin{align*}
\left| \PP(W \in A) - \PP(V \in A) \right| \leq \big(42d^{1/4} + 16
\big) \sum_{i=1}^\numobs \EE\|X_i\|_2^3.
\end{align*}
valid for any measurable convex set $A \subset \RR^d$.
\end{theorem}

In our analysis, it is convenient to make use of the following
consequence of this theorem.  Let $\{X_i\}_{i=1}^\numobs$ be a
sequence of zero-mean independent random vectors, and define
$\Sigma_\numobs \defn \frac{1}{\numobs} \sum_{i=1}^\numobs \EE[X_i
  X_i^T]$.  Then there is a universal constant $c$ such that for any
measurable convex set $A \subset \RR^d$, we have
\begin{align}
\label{EqnBerryConsequence}  
\left| \PP(\sqrt{\numobs} \,\Sigma_{\numobs}^{-1/2} \widebar{X} \in A)
- \PP(V \in A) \right| \leq \frac{cd^{1/4}}{\sqrt{\numobs}} \cdot
\frac{1}{\numobs} \sum_{i=1}^\numobs \EE\| \Sigma_\numobs^{-1/2} X_i
\|_2^3.
\end{align}


\subsection{Sample variance lemma}

We state and prove a simple lemma involving variances and empirical
variances.

\begin{lemma}
\label{lems:sample-var}
Suppose we have independent, zero-mean, random variables $X_1, \ldots,
X_\numobs$ and let $\widebar{X} \defn \empavg X_i$.
\begin{enumerate}

\item[(a)] The following equality holds:
\begin{align*}
\frac{1}{\numobs-1} \sum_{i=1}^\numobs \left( X_i - \widebar{X}
\right)^2 = \frac{1}{\numobs-1} \sum_{i=1}^\numobs X_i^2 -
\frac{\numobs}{\numobs-1} (\widebar{X})^2 =
\frac{1}{\numobs(\numobs-1)} \sum_{i<j} (X_i - X_j)^2.
\end{align*}

\item[(b)] We have
\begin{align*}
\EE\left[\frac{1}{\numobs(\numobs-1)} \sum_{i<j} (X_i - X_j)^2\right]
= \frac{1}{\numobs} \sum_{i=1}^\numobs \EE[X_i^2].
\end{align*}

\end{enumerate}
\end{lemma}

\begin{proof}
For part (a), we have
\begin{align*}
\frac{1}{\numobs-1} \sum_{i=1}^\numobs \left( X_i - \widebar{X}
\right)^2 &= \frac{1}{\numobs-1} \sum_{i=1}^\numobs \left( X_i^2 -
2X_i \widebar{X} + \widebar{X}^2\right) \\ &= \frac{1}{\numobs-1}
\sum_{i=1}^\numobs X_i^2 - 2\widebar{X} \sum_{i=1}^\numobs X_i +
\frac{\numobs}{\numobs-1} \widebar{X}^2 \\ &= \frac{1}{\numobs - 1}
\sum_{i=1}^\numobs - \frac{\numobs}{\numobs - 1} \widebar{X}^2,
\end{align*}
and
\begin{align*}
\frac{1}{\numobs(\numobs-1)} \sum_{i<j} (X_i - X_j)^2 &=
\frac{1}{\numobs(\numobs-1)} \sum_{i < j} \left( X_i^2 - 2X_iX_j +
X_j^2\right) \\
& = \empavg X_i^2 - \frac{2}{\numobs(\numobs-1)} \sum_{i < j} X_i X_j
\\
& = \frac{1}{\numobs - 1} \sum_{i=1}^\numobs X_i^2 -
\frac{1}{\numobs(\numobs-1)} \left(\sum_{i=1}^\numobs X_i^2 + 2\sum_{i
  < j} X_i X_j \right) \\
& = \frac{1}{\numobs-1}\sum_{i=1}^\numobs X_i^2 -
\frac{\numobs}{\numobs - 1} \widebar{X}^2.
\end{align*}

For part (b), some algebra yields
\begin{align*}
\EE[(X_i - X_j)^2] = \EE[X_i^2] + \EE[X_j^2],
\end{align*}
using the independence of $X_i, X_j$. The claim then follows from
straightforward algebra.
\end{proof}


\section{Construction of a hard ensemble}
\label{AppHardEnsemble}

In this appendix, we describe a construction of an ensemble with
$\GamMat = \IdMat_\usedim$ and $\SigMat = \diag(1, \omega, \ldots,
\omega)$, so that $\Lmat = \SigMat$, as discussed
in~\Cref{SecHardEnsemble}.  We begin by drawing $(\covariate,
\instrument)$ pairs from the ensemble~\eqref{EqnEndo} with $\strinst =
1$ and arbitrary choices of $(\strendo, \plainstd)$.  We construct the
instrument vector $\instrument \in \real^\usedim$ from a mixture
distribution as follows: letting $V \sim N(0, \IdMat_\usedim)$ and
$e_1 \in \real^\usedim$ be the standard basis vector with one in entry
$1$, we set
\begin{align*}
\instrument & = \begin{cases} \sqrt{2} \: \inprod{e_1}{V} \: e_1 &
  \mbox{with probability $1/2$, and} \\
  \sqrt{2} \: \big \{ V - \inprod{e_1}{V} e_1 \big \} & \mbox{with
    probability $1/2$.}
\end{cases}
\end{align*}
This construction ensures that $\instrument$ is zero-mean, with
$\Exs[\instrument \instrument^T] = \IdMat_\usedim$, and moreover that
$\Exs[\instrument \covariate^T] = \IdMat_\usedim$.

Letting $\Vtil_i \sim N(0,1)$ be standard Gaussian, we then
constructed the noise vector $\noise \in \real^\usedim$ with with
independent entries of the form $\noise_i = \sigma(\instrument_i)
\Vtil_i$ \mbox{for $i = 1, \ldots, \numobs$,} where the standard
deviation function takes the form
\begin{align*}
\sigma(\instrument_i) & = 1 \Ind[\instrument_{i1} \neq 0] + \omega
\Ind[\instrument_{i1} = 0] 
\end{align*}
for some weight parameter $\omega \geq 0$.  This construction ensures
that
\begin{align*}
\SigMat & = \Exs \Big[ \noise^2 \instrument \instrument^T \Big] =
\frac{1}{2} \big \{ e_1 e_1^T \big \} + \frac{\omega}{2} \big \{2
\IdMat - 2 e_1 e_1^T \big \} \; = \; \mbox{diag} \big(1, \omega,
\omega, \ldots, \omega).
\end{align*}

\section{Empirical validation of PM2.5 analysis}

In this section, we provide a detailed description of some validation
checks performed to substantiate our empirical findings. As described
earlier, we exhibited substantial agency in the construction of the
wildfire instrument, as our initial dataset consisted of all the fires
that occurred within the US in 2016 and we needed to construct a
one-dimensional instrument for each tract representing its exposure to
wildfire smoke. In this section, we provide the reported estimates of
effects for a variety of different instruments constructed, to be
described in the sequel. Table~\ref{table:iv-results-robust} validates
our results. For almost every combination of instrument and response,
our results returns statistically significant, positive estimates of
the effect of \pollm~pollution on various poor health outcomes; the
only exception being cancer in IV-5 with a $p$-value of $0.08$, which
is close to significant. For the most part, the reported effects
remain relatively the same with a few exceptions, namely asthma and
poor mental health in IV-5, and arthritis in IV-6.

The first row, IV, is just the results reported in
Table~\ref{table:iv-results}. The rows IV-2, and IV-3, are instruments
that are computed with the same formula as IV~\eqref{eqn:inst-defn},
just with different choice of thresholds. The row IV-4 is a modified
formula of~\Cref{eqn:inst-defn}, where we adjust the weights to be
higher due to the general flow of climate from west to east. More
precisely, the formula is given by
\begin{align*}
Z_i = \mathbf{1} \{ Z_i^* \geq c \}, \qquad \text{where} \qquad Z_i^*
= \frac{1}{|\mathcal{D}|} \sum_{f \in \mathcal{D}}
\frac{f_{\text{size}} \cdot (1 + w \mathbf{1}\{ \text{$f$ is west of
    $i$} \})}{(f_{\text{distance to $i$}})^2}.
\end{align*}
where $w$ is some user-chosen weight larger than $1$. IV-5 and IV-6
are based on a different formula for constructing the instrument
\begin{align*}
Z_i = \mathbf{1} \{ Z_i^* \geq c \}, \qquad \text{where} \qquad Z_i^*
= \frac{1}{|\mathcal{D}|} \sum_{f \in \mathcal{D}}
\frac{f_{\text{size}}}{f_{\text{distance to $i$}}}.
\end{align*}
The instrument IV-5 is given by $Z_i$, and IV-6 is given by
$Z_i^*$. In every scenario, the reported first-stage $F$-statistics
was always at least $100$, so we are not concerned with weak
instruments.

\begin{table}
\caption{Instrumental variable estimates using various different
  instruments and their standard errors.}
\label{table:iv-results-robust}
\begin{center}
\begin{tabular}{ l @{\qquad}| c @{ \qquad} c @{\qquad} c @{\qquad} c @{\qquad} c @{\qquad} c  }
\hline \hline & Arthritis & Asthma & Cancer & Heart disease & PMH & Stroke \\ 
\hline IV & $12.2$ & $5.2$ & $1.3$ & $2.9$ & $5.6$ & $1.5$
\\        & $(1.3)$ & $(0.5)$ & $(0.3)$ & $(0.4)$ & $(0.7)$ & $(0.2)$
\\ \\ IV-2 & $6.5$ & $3.4$ & $0.7$ & $1.8$ & $4.0$ & $0.9$ \\ &
$(1.0)$ & $(0.4)$ & $(0.3)$ & $(0.4)$ & $(0.6)$ & $(0.2)$ \\ \\ IV-3 &
$8.8$ & $4.3$ & $1.4$ & $2.1$ & $5.0$ & $1.0$ \\ & $(0.9)$ & $(0.4)$ &
$(0.3)$ & $(0.3)$ & $(0.5)$ & $(0.2)$ \\ \\ IV-4 & $12.3$ & $5.4$ &
$1.3$ & $2.8$ & $5.7$ & $1.5$ \\ & $(1.3)$ & $(0.5)$ & $(0.3)$ &
$(0.4)$ & $(0.7)$ & $(0.2)$ \\ \\ IV-5 & $9.3$ & $9.9$ & $0.7$ & $4.5$
& $12.6$ & $2.1$ \\ & $(1.6)$ & $(1.1)$ & $(0.4)$ & $(0.6)$ & $(1.5)$
& $(0.3)$ \\ \\ IV-6 & $4.3$ & $2.1$ & $2.4$ & $2.3$ & $6.1$ & $1.3$
\\ & $(0.8)$ & $(0.3)$ & $(0.3)$ & $(0.3)$ & $(0.6)$ & $(0.2)$
\\ \hline
\end{tabular}
\subcaption*{Results are reported in \% increase per 10 \pollm}
\end{center}
\end{table}


\end{document}